\documentclass[reqno]{amsart}
\usepackage{amsmath}%
\usepackage{amsfonts}%
\usepackage{amssymb}%
\usepackage{lscape}%
\usepackage[matrix,arrow,curve]{xy}%
\usepackage{graphicx}
\usepackage{tabularx}
\usepackage{cite}
\usepackage{url}
\usepackage{cmap}

\numberwithin{equation}{section}

\theoremstyle{plain}

\newtheorem{lemma}{Lemma}[section]

\theoremstyle{definition}
\newtheorem{remark}{Remark}

\def\Z{\mathbb{Z}}
\def\R{\mathbb{R}}
\def\C{\mathbb{C}}
\def\T{\mathbb{T}}

\def\coef{\operatorname{\mathfrak{cf}}}
\def\wt{\widetilde}
\def\wh{\widehat}
\def\bL{\boldsymbol L}
\def\bI{\boldsymbol I}

\def\Fr{\mathcal{F}}

\def\np{\operatorname{np}}
\def\re{\operatorname{Re}}
\def\quo{\hskip 0pt\lower 0.5 pt\hbox{${}^1\!\!/\!{}_4$}}

\def\tabErr{1}
\def\tabEmp{2}
\def\tabEmpCseven{3}

\begin{document}

\input epsf

\author{S.\,Yu.~Orevkov}
\address{IMT, l'universit\'e Paul Sabatier, Toulouse, France; Steklov Math. Inst., Moscow, Russia.}

\title[Asymptotics of the number of primitive lattice triangulations]
{Asymptotics of the number of primitive lattice triangulations of rectangles of width $4$ and $5$}

\maketitle

\begin{abstract}
Let $f(m,n)$ be the number of primitive lattice triangulations of an $m \times n$ rectangle. 
We express the limits $\lim_n f(m,n)^{1/n}$ for $m = 4$ and $m=5$ in terms of certain
systems of Fredholm integral equations on generating functions (the case $m\le3$ was treated in
a previous paper). Solving these equations numerically, 
we compute approximate values of these limits with a rather high precision.
\end{abstract}

\section{Introduction. Main results}\label{sect.intro}
This paper is a sequel of \cite{refOre2022}.
A {\it lattice triangulation} of a polygon in $\R^2$ is a triangulation
with all vertices in $\Z^2$.
A lattice triangulation is called {\it primitive} (or {\it unimodular})
if each triangle is primitive, i.e., has the minimal possible area $1/2$.
We denote the number of primitive lattice triangulations of an
$m\times n$ rectangle by $f(m,n)$. Let
$$
         c_{m,n} = \frac{ \log_2f(m,n)}{mn},
\qquad
         c_m = \lim_{n\to\infty} c_{m,n},
\qquad
         c = \lim_{m\to\infty} c_m
           = \lim_{n\to\infty} c_{n,n}.
$$
The existence of the limits is proven in \cite[Proposition 3.6]{refKZ}.
In \cite{refKZ} the number $c(m,n)$ is called the {\it capacity} of an $m\times n$ rectangle. 
Some estimates of $c_m$ and $c$ were obtained in
\cite{refAnclin,refKZ,refMVW,refOre1999,refOre2022,refWelzl}
(see more details in \cite{refOre2022}). In particular, the best known upper bound for $c$
is $c\le4\log_2\frac{1+\sqrt 5}2\approx 2.777$; see \cite{refWelzl}.

It is easy to show that $f(1,n)=\binom{2n}n$ and hence $c_1=2$; see \cite[\S 2.1]{refKZ}.
The exact value $c_2=\frac12\log_2\frac{611+\sqrt{73}}{36}$ and 360 decimal digits of $c_3$ were computed in \cite{refOre2022}.
Using the same approach, we compute here $65$ digits of $c_4$ and $15$ digits of $c_5$, namely,
$$
  \lim_{n}f(4,n)^{\frac 1{4n}} =
  4.29876 67575
    00969 11161
    79591 31117
    46998 15749
    21782 24986
    28476 37456
    15251,
$$
$$
  \lim_{n\to\infty}f(5,n)^{\frac 1{5n}} = 4.34096 16193 1753,
$$
and hence
$$
c_4 = 2.10392 28346
        93077 90885
        50919 47650
        35290 59916
        33019 91475
        08947 02758
        17980,
$$
$$ c \ge c_5 = 2.11801 46670 3561. $$
So far this is the best proven lower bound for $c$.

As in \cite{refOre2022},
we express $c_4$ (see \S\ref{sect.4}) and $c_5$ (see \S\ref{sect.5}) in terms of
solutions of systems of integral equations, and we compute the approximate values of $c_4$ and $c_5$
by solving systems of usual linear equations obtained by replacing the integrals with Riemann sums.
However, the computations here are more complicated than in \cite{refOre2022}.

Another difference from \cite{refOre2022} is the following.
In \cite{refOre2022} we reduced the problem to a single classical Fredholm equation of the first kind with an analytic
kernel of integration. Since the integral operator was compact, it was easy to prove the uniqueness of
its solution and the exponential rate of convergence of its discretizations.
In the present paper, the integral operators are no longer compact: the unknown functions are functions
of two (for $c_4$) or three (for $c_5$) variables but the operators involve the integration with
respect to only one variable (i.e., the kernel of integration is a kind of delta-function).
In \S\ref{sect.fred4} we discuss the uniqueness of solutions and the convergence of
discretizations for such operators.

In \S\ref{sect.compute} we expose two simple observations which allowed us to improve
the convergence. They were crucial for the computation of $c_4$ and $c_5$ with so high accuracy.
They also allowed us to compute $c_3$ to 1100 digits.
In \S\ref{sect.exact} we report about newly computed exact values of $f(m,n)$ and about an empirical estimate for $c_6$,
$c_7$, and for the subexponential factor of the asymptotics,
which we computed via known values of $f(m,n)$ by a method proposed by Lando and Zvonkin in \cite[\S6]{refLZ}.
In \S\ref{sect.np} we give an asymptotic lower bound for the number of all (not only primitive) lattice triangulations.

\subsection*{Acknowledgements}
I thank V.\;I.~Bogachev and S.\;K.~Lando for useful discussions.


\section{Strips of width 4}\label{sect.4}

In this section we express $c_4$ via solution of a certain system of integral equations
on the generating functions, whose numerical solution allows us to compute $c_4$ to
65 decimal digits.


\subsection{ Reduction to trapezoids (which decreases the dimension of the problem)}\label{sect.trapez4}
(Cf. \cite[\S4.1]{refOre2022}.)
For $a,e\ge 0$ such that $a+e$ is even, let $j_{a,e}^*$ be the number of primitive triangulations
of the trapezoid $T_4(a,e)$ spanned by $(0,0)$, $(1,4)$, $(1+e,4)$, $(a,0)$. We set $j_{0,0}^*=1$ and
$j^*_{a,e}=0$ when $a+e$ is odd. Consider the generating function
$$
   J^*(x) = \sum_n j_n^* x^n
          = \sum_{a,e\ge 0} j_{a,e}^* x^{a+e} = 1 + 6x^2 + 750 x^4 + 189121 x^6 + \dots
$$
Let $j_{a,e}$ be the number of the primitive triangulations of $T_4(a,e)$ which do not have interior
edges of the form $[(k,0),(l,4)]$ (by convention, $j_{0,0}=0$) and let
$$
   J(x) = \sum_n j_n x^n
        = \sum_{a,e\ge 0}j_{a,e} x^{a+e} = 6x^2 + 714 x^4 + 180337 x^6 + \dots
$$
Then we have
$$
     J^*(x) = \frac1{1-J(x)}\qquad\text{and}\qquad
     \lim_{n\to\infty}f(4,n)^{1/n} = \lim_{n\to\infty} (j_{2n}^*)^{1/n}
$$
(see \cite[\S4.1]{refOre2022} and \cite[Lemma 3.1]{refOre2022} respectively).
Thus
\begin{equation}\label{eq.trapez4}
                                     c_4 = -\frac12\log_2\beta_4,
\end{equation}
where $\beta_4$ is the first real root of the equation $J(x)=1$.



\subsection{ Recurrence relations }\label{sect.rec4}
We introduce a notation similar to that in \cite[\S4]{refOre2022}.
For integers $a,b,c,d,e$ such that $a\ge-1$ and $b,c,d,e\ge0$, let
$f_{a,b,c,d,e}$ denote the number of primitive lattice triangulations of the
polygon $[(0,a),(1,b),(2,c),(3,d),(4,e),(4,0),(0,-1)]$ (see Figure~\ref{figFGHJ}) without
interior edges of the form $[(0,y_1),(4,y_2)]$ ({\it side-to-side} edges).
Following \cite{refKZ}, we call these polygons {\it shapes}.
Similarly we define $g,h,j$ with corresponding subscripts and superscripts according
to Figure~\ref{figFGHJ}. For example, if $a\ge-1$, $\min(c,d,e)\ge0$, and $c-a\equiv 1$ mod $2$,
then we define $g_{a,c,d,e}^{(1)}$ to be the number of primitive lattice triangulations
of the shape $[(0,a),(2,c),(3,d),(4,e),(4,0),(0,-1)]$ without interior side-to-side edges.
If the inequalities or congruences are not satisfied, we set the corresponding numbers to zero.
By convention, $j_{-1,0}^{(1)}=0$; in this case the shape degenerates to the segment
$[(0,-1),(4,0)]$.

\begin{figure} 
	\centering
		\includegraphics[width=115mm]{fg.eps} 
\centerline{
 $f_{a,b,c,d,e}$  \hskip 37mm
  $g^{(1)}_{a,c,d,e}$ \hskip 13mm
  $g^{(2)}_{a,b,d,e}$ \hskip 13mm
  $g^{(3)}_{a,b,c,e}$ \hskip 2mm
}
\vskip3pt
\centerline{%
  \hskip 48mm $c-a\equiv 1(2)$
  \hskip  5mm $d-b\equiv 1(2)$
  \hskip  5mm $e-c\equiv 1(2)$
}
\centering
		\includegraphics[width=115mm]{h.eps} 
\centerline{
  $h^{(1)}_{a,d,e}$ \hskip 15mm
  $h^{(2)}_{a,d,e}$ \hskip 15mm
  $h^{(3)}_{a,c,e}$ \hskip 15mm
  $h^{(4)}_{a,b,e}$ \hskip 15mm
  $h^{(5)}_{a,b,e}$ \hskip 2mm
}
\vskip3pt
\centerline{%
  \hskip 0mm $d-a\equiv 1(3)$
  \hskip 5mm $d-a\equiv 2(3)$
  \hskip 5mm $c-a\equiv 1(2)$
  \hskip 5mm $e-b\equiv 1(3)$
  \hskip 5mm $e-b\equiv 2(3)$
}
\centerline{$\,e-c\equiv 1(2)$}
\centering
		\includegraphics[width=115mm]{j.eps} 
\centerline{%
  $j^{(1)}_{a,e}$  \hskip 17mm
  $j^{(2)}_{a,e}$  \hskip 25mm
}
\vskip3pt
\centerline{%
  \hskip 0mm $e-a\equiv 1(4)$
  \hskip 4.5mm $e-a\equiv 3(4)$ \hskip 20mm
}
	\caption{}
	\label{figFGHJ}
\end{figure}

We introduce the generating functions for these numbers
$$
   F(x,y,z,w,v) = \sum_{a,b,c,d,e} f_{a,b,c,d,e} x^a y^b z^c w^d v^e,
\qquad
   J_k(x,v) = \sum_{a,e} j^{(k)}_{a,e} x^a v^e\quad (k=1,2),
$$
$$
   G_1(x,z,w,v) = \sum_{a,c,d,e} g^{(1)}_{a,c,d,e} x^a z^c w^d v^e,
\qquad
   G_2(x,y,w,v) = \sum_{a,b,d,e} g^{(2)}_{a,b,d,e} x^a y^b w^d v^e,
$$
$$
   G_3(x,y,z,v) = \sum_{a,b,c,e} g^{(3)}_{a,b,c,e} x^a y^b z^c v^e,
\qquad
   H_3(x,z,v) = \sum_{a,c,e} h^{(3)}_{a,c,e} x^a z^c v^e,
$$
$$
   H_k(x,w,v) = \sum_{a,d,e} h^{(k)}_{a,d,e} x^a w^d v^e\; (k=1,2), \quad
   H_k(x,y,v) = \sum_{a,b,e} h^{(k)}_{a,b,e} x^a y^b v^e\; (k=4,5),
$$
and their symmetrizations
\begin{align*}
   \wt F(x,y,z,w,v) &= F(x,y,z,w,v) + F(v,w,z,y,x),\\
   \wt G_1(x,z,w,v) &= G_1(x,z,w,v) + G_3(v,w,z,x),\\
   \wt G_2(x,y,w,v) &= G_2(x,y,w,v) + G_2(v,w,y,x),\\
   \wt H_k(x,w,v) &= H_k(x,w,v) + H_{6-k}(v,w,x),\qquad k=1,2,\\
   \wt H_3(x,z,v) &= H_3(x,z,v) + H_3(v,z,x),\\
   \wt J_1(x,v) &= J_1(x,v) + J_2(v,x).
\end{align*}

Throughout the paper we use the notation $\coef_{\mathbf{m}}\mathcal F$ for the coefficient
of a monomial $\mathbf{m}=\mathbf{m}(x_1,x_2,\dots)$ in a Laurent series
$\mathcal F=\mathcal F(x_1,x_2,\dots)$.
In this notation, the recurrence relations in \cite[Lemma 2.2]{refOre2022} take the
following form (cf.~\cite[\S4.2]{refOre2022}; here we omit the intermediate step
consisting in finding the relations between the non-symmetrized generating functions):
\begin{align*}
 \wt F(x,y,z,w,v)Q(x,y,z,w,v)
     & =      y^{1/2}\wt G_1\big(xy^{1/2},y^{1/2}z,w,v\big)(1 - w - v) \\
     &\quad + z^{1/2}\wt G_2\big(x,yz^{1/2},z^{1/2}w,v\big)(1 - x)(1 - v) \\
     &\quad + w^{1/2}\wt G_1\big(vw^{1/2},w^{1/2}z,y,x\big)(1 - y - x) \\
     &\quad - y^{1/2}w^{1/2}\wt H_3\big(xy^{1/2},y^{1/2}zw^{1/2},w^{1/2}v\big),
\end{align*}
where $Q(x,y,z,w,v) =  1 - x - y - z - w - v + xz + xw + xv + yw + yv + zv - xzv$,
\begin{align*}
   \wt G_1(x,z,w,v)(1-w-v)
         &= \coef_{\xi^{-1}}\wt F\Big(\frac{x}\xi,\xi^2,\frac{z}\xi,w,v\Big)(1-w-v) \\
         &\quad - w^{1/2}\coef_{\xi^{-1}}\wt G_1\Big(
                         vw^{1/2},\frac{w^{1/2}z}\xi,\xi^2,\frac{x}\xi\Big)\\
         &\quad + z^{1/3}(1-v)\wt H_1\big(xz^{1/3},z^{2/3}w,v\big)\\
         &\quad + w^{1/2}\wt H_3\big(x,zw^{1/2},w^{1/2}v\big),
\end{align*}
\begin{align*}
  \wt G_2(x,y,w,v)(1-x)(1-v)
          &= \coef_{\xi^{-1}}\wt F\Big(x,\frac{y}\xi,\xi^2,\frac{w}\xi,v\Big)(1-x)(1-v)\\
          &\quad + y^{1/3}(1-v)\wt H_2\big(xy^{2/3},y^{1/3}w,v\big)\\
          &\quad + w^{1/3}(1-x)\wt H_2\big(vw^{2/3},w^{1/3}y,x\big),
\end{align*}
\begin{align*}
 &\wt H_1(x,w,v)(1-v) = \coef_{\xi^{-1}}\wt G_2\Big(\frac{x}{\xi^2},\xi^3,\frac{w}\xi,v\Big)(1-v)
              + \frac{1}{x},
\\
 &\wt H_2(x,w,v) = \coef_{\xi^{-1}}\wt G_1\Big(\frac{x}\xi,\xi^3,\frac{w}{\xi^2},v\Big),
\\
 &\wt H_3(x,z,v)
          = \coef_{\xi^{-1}}\!\wt G_1\Big(x,\frac{z}\xi,\xi^2,\frac{v}\xi\Big)
          + \coef_{\xi^{-1}}\!\wt G_1\Big(v,\frac{z}\xi,\xi^2,\frac{x}\xi\Big) 
   - \coef_{\xi_1^{-1}\xi_2^{-1}}\!\wt F\Big(
              \frac{x}{\xi_1},\xi_1^2,\frac{z}{\xi_1\xi_2},\xi_2^2,\frac{v}{\xi_2}\Big),\\
  &\wt J_1(x,v) = \coef_{\xi^{-1}}\wt H_2\Big(\frac{v}\xi,\xi^4,\frac{x}{\xi^3}\Big).
\end{align*}
Notice that 
(see \S\ref{sect.trapez4} for the definition of $J(x)$)
\begin{equation}\label{eq.J}
                                   J(x) = x\wt J_1(x,x).
\end{equation}


\subsection{ A change of variables }\label{sect.ch.var4}
We set
\begin{align*}
  f(t,s,u) = \wt F\Big(\frac{x}{t},\frac{x^2t^2}{s},\frac{x^2s^2}{tu},\frac{x^2u^2}{s},\frac{x}{u}\Big),&\qquad
  q(t,s,u) = Q\Big(\frac{x}{t},\frac{x^2t^2}{s},\frac{x^2s^2}{tu},\frac{x^2u^2}{s},\frac{x}{u}\Big),\\
  g_1(s,u) = \frac{1}{s^{1/2}}\,
    \wt G_1\Big(\frac{x^2}{s^{1/2}},\frac{x^3 s^{3/2}}{u},\frac{x^2 u^2}{s},\frac{x}{u}\Big),&\quad
  g_2(t,u) = \frac{1}{(tu)^{1/2}}\,
    \wt G_2\Big(\frac{x}t,\frac{x^3t^{3/2}}{u^{1/2}},\frac{x^3u^{3/2}}{t^{1/2}},\frac{x}u\Big),\\
  h_k(u) = \frac{1}{u^{k/3}}\wt H_k\Big(\frac{x^3}{u^{1/3}}, x^4 u^{4/3}, \frac{x}{u}\Big)\;\;(k=1,&2),\qquad
  h_3(s) = \wt H_3\Big(\frac{x^2}{s^{1/2}}, x^4 s, \frac{x^2}{s^{1/2}}\Big),\\
  j_1(x)=\wt J_1(x^4,x^4), \qquad\quad
  p(s,u) = 1\, - \,&\frac{x^2u^2}{s} - \frac{x}{u}, \qquad\quad r(t,u)=\Big(1-\frac{x}{t}\Big)\Big(1-\frac{x}{u}\Big).
\end{align*}

The combinatorial meaning of $f,g_k,h_k,j_1$ is the following.
Let $S$ be a shape contributing to the coefficient of a monomial $\mathbf m$ of one of these series. 
Then the exponent of $x$ in $\mathbf m$ is equal to $2\int_0^4\varphi(x)dx$ where $y=\varphi(x)$ is the equation of
the upper boundary of $S$. The variables $t,s,u$ correspond to the vertical lines $x=1$, $x=2$, $x=3$
respectively. If $(x_0,y_0)$, $(x_1,y_1)$, $(x_2,y_2)$ are three consecutive vertices on the upper boundary of $S$,
then the exponent of the variable corresponding to the vertical line $x=x_1$ is equal to the
integer part of the difference of the slopes of the adjacent segments, that is
$$
     \left\lfloor\frac{y_1-y_0}{x_1-x_0} - \frac{y_2-y_1}{x_2-x_1}\right\rfloor.
$$
In particular, $\mathbf m$ and the congruences in Figure~\ref{figFGHJ} determine the upper boundary of $S$
up to automorphism of $\mathbb Z^2$ of the form $(x,y)\mapsto(x,y+ax+b)$.
For example, the shapes for
$f_{3,2,3,1,2}$ and $g^{(1)}_{3,2,3,2}$
depicted in Figure~\ref{figFGHJ} contribute to
the coefficients of the monomials $x^{17}t^{-1}s^3u^{-3}$ and $x^{20}s^{-2}u^2$ in the series $f$ and $g_1$ respectively.

It follows that $f,g_k,h_k,j_1$ are elements of the ring
$$
    \Z[t^{\pm1},s^{\pm1},u^{\pm1}]((x))
$$
of formal power series in $x$ whose coefficients are Laurent polynomials in $t,s,u$.
The division of such power series will be understood as the division in this ring.

Applying the relation from \S\ref{sect.rec4} we obtain
\begin{align*}
f(t,s,u)\,q(t,s,u)
     &= \frac{xt}{s^{1/2}}\,\wt G_1\Big(\frac{x}{t}\,\frac{xt}{s^{1/2}},\;\frac{xt}{s^{1/2}}\,\frac{x^2s^2}{tu},\;\frac{x^2u^2}{s},\;
                                                                                                                   \frac{x}{u}\Big)\,p(s,u) \\
     &\quad + \frac{xs}{(tu)^{1/2}}\,\wt G_2\Big(\frac{x}{t},\,\frac{x^2t^2}{s}\,\frac{xs}{(tu)^{1/2}},\,\frac{xs}{(tu)^{1/2}}\,
                                                                                                 \frac{x^2u^2}{s},\,\frac{x}{u}\Big)\,r(t,u) \\
     &\quad + \frac{xu}{s^{1/2}}\,\wt G_1\Big(\frac{x}{u}\,\frac{xu}{s^{1/2}},\,\frac{xu}{s^{1/2}}\,\frac{x^2s^2}{tu},\,
                                                                                                 \frac{x^2t^2}{s},\,\frac{x}{t}\Big)\,p(s,t) \\
     &\quad - \frac{xt}{s^{1/2}}\,\frac{xu}{s^{1/2}}\,\wt H_3\Big(\frac{x}{t}\,\frac{xt}{s^{1/2}},\,
                  \frac{xt}{s^{1/2}}\,\frac{x^2s^2}{tu}\,\frac{xu}{s^{1/2}},\,\frac{xu}{s^{1/2}}\,\frac{x}{u}\Big).
\end{align*}
Given any Laurent series $\mathcal F(\xi,x_1,x_2,\dots)$, a monomial $\mathbf{m}=\mathbf{m}(x_1,x_2,\dots)$,
and a new variable $t$, we have (cf. \cite[Eq. (9)]{refOre2022})
\begin{equation}\label{change.var}
   \coef_{\xi^{-1}}\mathcal F(\xi,x_1,x_2,\dots) = \mathbf m\coef_{t^{-1}}t^{\alpha-1}\mathcal F(t^\alpha\mathbf m,x_1,x_2,\dots).
\end{equation}
Using \eqref{change.var} with the substitutions $\xi=xt/s^{1/2}$ (in $g_1$),  $\xi=xs/(tu)^{1/2}$ (in $g_2$),
$\xi=xt^{1/2}/u^{1/6}$ (in $h_1$), $\xi=xs^{1/2}/u^{1/3}$ (in $h_2$),
$\xi_1=xt/s^{1/2}$, $\xi=\xi_2=xu/s^{1/2}$ (in $h_3$), and $\xi=xu^{1/3}$ \hbox{(in $j_1$),}
we obtain
\begin{align*}
g_1(s,u)p(s,u)
   &= \frac1{s^{1/2}}\bigg\{p(s,u)\,\frac x{s^{1/2}}\coef_{t^{-1}}\wt F\Big(\frac{x^2}{s^{1/2}}\,\frac{s^{1/2}}{xt},
               \,\frac{x^2t^2}{s},\,\frac{s^{1/2}}{xt}\,\frac{x^3 s^{3/2}}{u},\,\frac{x^2u^2}{s},\frac{x}{u}\Big) \\
   &\quad -\frac{xu}{s^{1/2}}\,\frac x{s^{1/2}} \coef_{t^{-1}}\wt G_1\Big(\frac{x}{u}\,\frac{xu}{s^{1/2}},\,\frac{xu}{s^{1/2}}
            \,\frac{x^3s^{3/2}}{u}\,\frac{s^{1/2}}{xt},\,\frac{x^2t^2}s,\,\frac{s^{1/2}}{xt}\,\frac{x^2}{s^{1/2}}\Big)\\
   &\quad +\frac{xs^{1/2}}{u^{1/3}}\Big(1-\frac x u\Big)\wt H_1\Big(\frac{x^2}{s^{1/2}}\,\frac{xs^{1/2}}{u^{1/3}},\,
                                                            \frac{x^2s}{u^{2/3}}\,\frac{x^2u^2}{s},\,\frac x u\Big)\\
   &\quad
    +\frac{xu}{s^{1/2}}\wt H_3\Big(\frac{x^2}{s^{1/2}},\,\frac{x^3s^{3/2}}{u}\,\frac{xu}{s^{1/2}},\,\frac{xu}{s^{1/2}}\,
    \frac x u\Big)\bigg\},
\end{align*}
\begin{align*}
g_2(t,u)r(t,u)
  &=\!\frac1{\!(tu)^{1/2}\!}\bigg\{ r(t,u)\frac x{(tu)^{1/2}\!\!}\coef_{s^{-1}}\!\wt F\Big(\frac x t,\frac{x^3t^{3/2}}{u^{1/2}}\,
     \frac{(tu)^{1/2}}{xs}, \frac{x^2s^2}{tu}, \frac{(tu)^{1/2}}{xs}\,\frac{x^3u^{3/2}}{t^{1/2}}, \frac x u\Big)\\
  &\quad + \frac{xt^{1/2}}{u^{1/6}}\Big(1-\frac x u\Big)\wt H_2\Big(\frac x t\,\frac{x^2t}{u^{1/3}},\,\frac{xt^{1/2}}{u^{1/6}}\,
                   \frac{x^3u^{3/2}}{t^{1/2}},\,\frac x u\Big)\\
  &\quad + \frac{xu^{1/2}}{t^{1/6}}\Big(1-\frac x t\Big)\wt H_2\Big(\frac x u\,\frac{x^2u}{t^{1/3}},\,\frac{xu^{1/2}}{t^{1/6}}\,
          \frac{x^3t^{3/2}}{u^{1/2}},\,\frac x t\Big)\bigg\},
\end{align*}
$$
h_1(u) = \frac1{u^{1/3}}\left\{\frac x{u^{1/6}} \coef_{t^{-1}} \frac1{t^{1/2}}\wt G_2\Big(\frac{x^3}{u^{1/3}}\,\frac{u^{1/3}}{x^2t},\,
                 \frac{x^3t^{1/3}}{u^{1/2}},\,\frac{u^{1/6}}{xt^{1/2}}\,x^4u^{4/3},\,\frac x u\Big) + \frac{u^{1/3}}{x^3(1-x/u)} \right\},
$$
$$
h_2(u) = \frac1{u^{2/3}}\left\{\frac x{u^{1/3}}\coef_{s^{-1}}\frac1{s^{1/2}}\wt G_1\Big(\frac{x^3}{u^{1/3}}\,
         \frac{u^{1/3}}{xs^{1/2}},\, \frac{x^3s^{3/2}}u,\,\frac{u^{2/3}}{x^2s}\,x^4u^{4/3},\,\frac x u\Big)\right\},
$$
\begin{align*}
h_3(s) &= \frac{2x}{s^{1/2}}\coef_{u^{-1}}\wt G_1\Big(\frac{x^2}{s^{1/2}},\, x^4s\,\frac{s^{1/2}}{xu},\,
                                             \frac{x^2u^2}s,\, \frac{s^{1/2}}{xu}\,\frac{x^2}{s^{1/2}}\Big)\\
   &\quad - \frac{x^2}s \coef_{t^{-1}u^{-1}}\wt F\Big(\frac{x^2}{s^{1/2}}\,\frac{s^{1/2}}{xt},\,\frac{x^2t^2}s,\,
          \frac{s^{1/2}}{xt}\,x^4s\,\frac{s^{1/2}}{xu},\,\frac{x^2u^2}s,\,  \frac{s^{1/2}}{xu}\,\frac{x^2}{s^{1/2}}\Big),
\end{align*}
$$
    j_1(x) = x\coef_{u^{-1}}\frac1{u^{2/3}}\wt H_2\Big(x^4\,\frac{1}{x u^{1/3}},\, x^4u^{4/3},\, \frac{1}{x^3u}\,x^4\Big).
$$
Thus,
\begin{equation}\label{t4.f.init}
f(t,s,u)\,q(t,s,u) = xt\, g_1(s,u)p(s,u) + xs\, g_2(t,u)r(t,u)+xu\, g_1(s,t)p(s,t)-\frac{x^2tu}{s} h_3(s),
\end{equation}
\begin{equation}\label{t4.g1.init}
g_1(s,u)\,p(s,u) = \frac x s\, p(s,u)\coef_{t^{-1}} f(t,s,u) - \frac{x^2u}s \coef_{t^{-1}} g_1(s,t)
                  + \frac{x(u-x)}u h_1(u) + \frac{xu}s h_3(s),
\end{equation}
\begin{equation}\label{t4.g2.init}
g_2(t,u) = \frac x{tu}\coef_{s^{-1}} f(t,s,u) + \frac{xt}{t-x}h_2(u) + \frac{xu}{u-x}h_2(t),
\end{equation}
\begin{equation}\label{t4.h1.init}
h_1(u) = x\coef_{t^{-1}}g_2(t,u) + \frac u{x^3(u-x)},
\end{equation}
\begin{equation}\label{t4.h2.init}
h_2(u) = \frac x u\coef_{s^{-1}}g_1(s,u),
\end{equation}
\begin{equation}\label{t4.h3.init}
h_3(s) = 2x \coef_{u^{-1}}g_1(s,u) - \frac{x^2}s\coef_{t^{-1}u^{-1}} f(t,s,u),
\end{equation}
\begin{equation}\label{t4.j1.init}
j_1(x) = x \coef_{u^{-1}}h_2(u).
\end{equation}


\subsection{ Elimination of $f$, $h_1$, and $h_2$ }\label{sect.elim4}

Let
\begin{xalignat*}{2}
  &\Phi_1(s,u) = \coef_{t^{-1}}\frac t{q(t,s,u)},            &&\Psi_1(s,u)=1 - \frac{x^2}s p(s,u)\Phi_1(s,u),\\
  &\Phi_2(t,u) = \coef_{s^{-1}}\frac s{q(t,s,u)},            &&\Psi_2(t,u)=1 - \frac{x^2}{tu}r(t,u)\Phi_2(t,u),\\
  &\Phi_{3}(s)   = \coef_{t^{-1}u^{-1}}\frac{tu}{q(t,s,u)}, &&\Psi_{3}(s) = 1 - \frac{x^4}{s^2}\Phi_{3}(s).
\end{xalignat*}

\begin{remark} A {\tt Wolfram Mathematica} code
that checks the identities \eqref{t4.g1} -- \eqref{t4.j1} below up to $O(x^n)$ is available at
  \url{https://www.math.univ-toulouse.fr/~orevkov/tr45.html}.
\end{remark}

We eliminate $f$, $h_1$, $h_2$ from \eqref{t4.f.init}--\eqref{t4.h3.init} by plugging
\eqref{t4.f.init}, \eqref{t4.h1.init}, \eqref{t4.h2.init} into
\eqref{t4.g1.init}, \eqref{t4.g2.init}, \eqref{t4.h3.init}.
After a simplification we obtain the following system of equations for $g_1$, $g_2$, and $h_3$:
\begin{equation}\label{t4.g1}
\begin{split}
   \Psi_1(s,u)g_1(s,u) &= \frac{x^2u}{s}\coef_{t^{-1}}\left(\frac{p(s,t)}{q(t,s,u)} - \frac1{p(s,u)}\right)g_1(s,t) \\
                       &\quad + \frac{x^2(u-x)}{u}\coef_{t^{-1}}\left(\frac{t-x}{t\,q(t,s,u)} + \frac1{p(s,u)}\right)g_2(t,u)\\
                       &\quad + \frac{xu\Psi_1(s,u)}{s\, p(s,u)} h_3(s) \;+\; \frac1{x^2 p(s,u)},
\end{split}
\end{equation}

\begin{equation}\label{t4.g2}
\begin{split}
   \Psi_2(t,u)g_2(t,u)
      &=  \frac{x^2}{u}\coef_{s^{-1}}\left(\frac{p(s,u)}{q(t,s,u)} + \frac t{t-x}\right)g_1(s,u) \\
      & + \frac{x^2}{t}\coef_{s^{-1}}\left(\frac{p(s,t)}{q(t,s,u)} + \frac u{u-x}\right)g_1(s,t)
      \; - \; \coef_{s^{-1}} \frac{x^3 h_3(s)}{s\,q(t,s,u)},
\end{split}
\end{equation}

\begin{equation}\label{t4.h3}
   \Psi_{3}(s)h_3(s) = 2x\coef_{u^{-1}} \Psi_1(s,u) g_1(s,u) \;-\; x^3\coef_{t^{-1}u^{-1}}\frac{r(t,u)}{q(t,s,u)}\,g_2(t,u).
\end{equation}

Eliminating $h_2(u)$ from \eqref{t4.h2.init} and \eqref{t4.j1.init} we express $j_1$ via $g_1$:
\begin{equation}\label{t4.j1}
                                j_1(x) = x^2 \coef_{s^{-1}u^{-1}} \frac{g_1(s,u)}u.
\end{equation}


\subsection{Estimates for the radii of convergence}\label{sect.converge4}
Let $\T$ be the unit circle in $\C$ centered at $0$.
For a series $\mathcal F\in\Z[t^{\pm1},s^{\pm1},u^{\pm1}]((x))$ we define its {\it $x$-radius of convergence} denoted by $R_x(\mathcal F)$
as the supremum of positive $x_0$ such that $\mathcal F$ converges in a neighborhood of the set
$[0,x_0]\times\T^3$ in $\C^4$. If all the coefficients of $\mathcal F$ are positive,
then $R_x(\mathcal F)$ coincides with the radius of convergence of $\mathcal F(x,1,1,1)$.

\begin{lemma}\label{lem.pos.coef4}
 All coefficients of $f$, $g_k$, $h_k$, $j_1$, $\Phi_k$,
 $1/q$, $1/p$, $1/r$, $1/(p\Psi_1)$, $1/(r\Psi_2)$, $1/\Psi_{3}$
 are positive,
 and $R_x(1/q)=1/2$, $R_x(1/p)=(-1+\sqrt5)/2\approx 0.618$, $R_x(1/r)=1$.
\end{lemma}

\begin{proof} Let ``Pos($\mathcal F$)'' mean ``all coefficients of $\mathcal F$ are positive''.
It is enough to prove  Pos($1/q$), Pos$(1/(p\Psi_1))$, and Pos$(1/(r\Psi_2))$ because
Pos$(1/q)\Rightarrow\text{Pos}(\Phi_k)$, Pos$(\Phi_3)\Rightarrow\text{Pos}(1/\Psi_3)$, and the statement
for the other series is evident.

\smallskip
Pos($1/q$).
Let $F_0(x,y,z,w,v)$ be defined as $F$ but counting only the primitive triangulation such that
the projection of any non-vertical edge to the horizontal axis has length 1, and the bottom of
the shapes is the segment $[(0,0),(4,0)]$. Let $f_0$ be obtained from $F_0$ by the substitutions in \S\ref{sect.ch.var4}.
Then (see \cite[Example 2.3]{refOre2022}) we have $F_0=1/Q$.
Hence $1/q=f_0$ has all positive coefficients.
Since $q(1,1,1)=(1-2x)(1-2x^2)$, we have $R_x(1/q)=1/2$.

\smallskip 
Pos$(1/(p\Psi_1))$. 
Let $F_1$ and $G_{11}$ be defined as $F$ and $G_1$ but counting only the primitive triangulations such that
the projection of any non-vertical edge to the horizontal axis is either the segment $[0,2]$ or a segment of length 1,
and the bottom of the shapes is the union of the segments $[(0,-1),(2,0)]\cup[(2,0),(4,0)]$.
Then we have (cf.~\S\ref{sect.rec4})
\begin{align*}
    F_1(x,y,z,w,v)Q(x,y,z,w,v) &= y^{1/2}G_{11}(xy^{1/2},y^{1/2}z,w,v)(1-w-v),\\
    G_{11}(x,z,w,v)(1-w-v) &= \coef_{\xi^{-1}} F_1\Big(\frac x\xi,\xi^2,\frac z\xi,w,v\Big)(1-w-v) + \frac1{x}.
\end{align*}
Let $f_1$ and $g_{11}$ be obtained
from $F_1$ and $G_{11}$ by the substitutions in \S\ref{sect.ch.var4}. Then the analogue of 
\eqref{t4.g1} takes the form $p\,\Psi_1\,g_{11}=1/x^2$, hence the coefficients of $1/(p\Psi_1)$ are positive.

\smallskip
Pos$(1/(r\Psi_2))$.
Let $F_2$ and
$G_{22}$ be defined as $F$ and 
$G_2$ but counting only the primitive triangulations such that
the projection of any non-vertical edge to the horizontal axis is either the segment $[1,3]$ or a segment of length 1, and the bottom of
the shapes is the polygonal chain $[(0,-1),(1,0),(3,1),(4,1)]$. Let $f_2$ and 
$g_{22}$ be obtained from 
$F_2$ and 
$G_{22}$ by the substitutions in \S\ref{sect.ch.var4}. Proceeding as in \S\S\ref{sect.rec4}--\ref{sect.ch.var4}, we obtain
$r\,\Psi_2\, g_{22}=x^3$, whence the result.
\end{proof}

\begin{remark}
All the coefficients of $1-\Psi_k$ and $1/\Psi_k$ that we computed are also positive.
\end{remark}

\begin{lemma}\label{lem.2roots.4}
(a). For any $x\in[0,\frac12)$ the function $q(t,s,u)$ does not vanish on $\T^3$.

(b).
For any fixed $(x,s,u)\in(0,\frac12)\times\T^2$, the equation $q(t,s,u)=0$
has two roots $t_1(x,s,u),t_2(x,s,u)$ in the disk $|t|<1$ and two roots in its complement.

For any fixed $(x,t,u)\in(0,\frac12)\times\T^2$, the equation $q(t,s,u)=0$
has two roots $s_1(x,t,u),s_2(x,t,u)$ in the disk $|s|<1$ and two roots in its complement.

(c).
One has
$$
  \Phi_1(s,u)
  = \frac1{2\pi i}\int_\T \frac{t\,dt}{q(t,s,u)}
   =  \frac{t_1}{q'_t(t_1,s,u)} + \frac{t_2}{q'_t(t_2,s,u)},
$$
$$
  \Phi_2(t,u)
  = \frac1{2\pi i}\int_\T \frac{s\,ds}{q(t,s,u)}
   =  \frac{s_1}{q'_s(t,s_1,u)} + \frac{s_2}{q'_s(t,s_2,u)},
$$
$$
   \Phi_{3}(s)
    = \frac1{(2\pi i)^2}\int_{\T^2}\frac{t\hskip 0.8pt u\,dt\,du}{q(t,s,u)}
    = \frac1{2\pi i}\int_\T \Phi_1(s,u)u\,du.
$$
\end{lemma}

\begin{proof} (a). Follows from Lemma \ref{lem.pos.coef4}.

(b). 
By (a), the number of roots in the unit disk is constant on $[0,\frac12)\times\T^2$, and
one easily checks that it is equal to $2$ at $(x,1,1)$ for a small $x$.

(c). Follows from (b) by the residue formula for the Cauchy integrals.
\end{proof}

\begin{remark}
The same arguments give a computation-free proof of Lemma 4.1 in \cite{refOre2022}.
\end{remark}

Recall that $R_x(J^*)$ is denoted in \S\ref{sect.trapez4} by $\beta_4$.

\begin{lemma}\label{lem.Rx4}
(a). The $x$-radii of convergence of the series $f$, $g_k$, $h_k$, $j_1$
are greater than $\beta_4^{1/4}$, and we have $\beta_4>0$.

(b). We have $R_x(\Phi_1)=R_x(\Phi_2)=R_x(\Phi_3)=1$ and
\begin{equation}                     \label{Rx(Psi)}
   R_x(\Psi_1)= 0.495375..., \qquad
   R_x(\Psi_2)= 0.495455..., \qquad
   R_x(\Psi_{3})= 0.499999...
\end{equation}
\end{lemma}

\begin{figure}\label{fig.Psi}
	\centering
		\includegraphics[width=66mm]{psi3-graph.eps}\hskip 5mm
		\put(-7,17){$x$} \put(-170,167){$y$}
		\put(-70,133){$\Psi_1$}
		\put(-50,138){$\Psi_2$}
		\put(-24,147){$\Psi_{3}$}
\vskip-1mm
	\caption {The graphs of $\Psi_1(x;1,1)$, $\Psi_2(x;1,1)$, $\Psi_{3}(x;1)$ on $\left[\frac14,\frac12\right]$.}
	\label{fig.Psi}
\end{figure}

\begin{proof}
(a).
The inequality $\beta_4>0$ follows from the exponential upper bound $f(m,n)<8^{mn}$.
By the arguments as in \cite[Lemma 3.2]{refOre2022}, it is easy to show that
the radius of convergence of each of $f(1,1,1)$, $g_k(1,1)$, $h_k(1)$, $j_1$ coincides with that of $J(x^4)$,
hence it is greater than $\beta_4^{1/4}$. Since these Laurent series have positive coefficients, the same
is true for the $x$-radii of convergence of $f$, $g_k$, $h_k$, $j_1$.

\smallskip
(b). We have $R_x(\Phi_k)\ge R_x(1/q)\ge 1/2$ by Lemma~\ref{lem.pos.coef4}.
Analyzing the behavior of the integrals in Lemma~\ref{lem.2roots.4}(c) when $x\to 1/2$ one can show
that $\Phi_k(1,1)\sim C_k(1-2x)^{-1/2}$, $k=1,2$, and
$\Phi_{3}(1)\sim C_3\log(1-2x)$ for some constants $C_k$. We omit the details since in fact we need only the
lower bounds for the radii of convergence.

The functions $p\Psi_1$, $r\Psi_2$, $\Psi_3$ with $t=s=u=1$
decrease on the interval $[0,1/2)$ by Lemma~\ref{lem.pos.coef4}
and they can be computed with any precision using Lemma~\ref{lem.2roots.4}(c).
Then it is easy to find numerically the zero of each of them on this interval;
see also Figure~\ref{fig.Psi}.
\end{proof}


\subsection{ The system of integral equations and its discretization }
\label{sect.fred4} 
Recall that $\beta_4$ is the first real root of the equation $J(x)=1$ (see \S\ref{sect.trapez4}).
The value of $c_4$ announced in the introduction corresponds
to $\beta_4= 0.054114\dots$ but our aim is to justify this computation and we do not assume yet that $\beta_4$ is
close to this value.
Let us set $\beta_4^+=0.05414$ (see Figure~\ref{fig.J}),
$x_0=\beta_4^{1/4}$, and $x_0^+=(\beta_4^+)^{1/4}\approx 0.482369$.

In the previous subsection we have shown that the series involved in the equations
\eqref{t4.g1}--\eqref{t4.h3} converge in a neighborhood of $\{x\}\times\T^3$
whenever $x<\min(x_0,x_0^+)$. Hence, for such $x$, we
may replace $\coef_{\xi^{-1}}(\dots)$ by $\frac{1}{2\pi i}\int_\T(\dots)d\xi$
(here $\xi$ stands for $t$, $s$, or $u$). Then for each $x$ we obtain a system of three
integral equations for the restrictions of $g_1$, $g_2$, and $h_3$ to $\T^2$ and $\T$.

We are going to prove that this system has a unique solution for any fixed $x\in[0,x_0^+]$
and the solution analytically depends on $x$. Then, by the identity theorem for analytic functions,
the series $g_1$, $g_2$, $h_3$ converge on this interval and,
replacing the integrals by Riemann sums, we can compute $g_1$ at any point by solving
numerically the resulting system of linear equations. This would allow us to find $J(x^4)$
(from \eqref{t4.j1} and \eqref{eq.J}) for any $x\in[0,x_0^+]$, and (since $J$ is monotone)
to solve the equation $J(x)=1$.

\begin{figure}
	\centering
		\includegraphics[width=66mm]{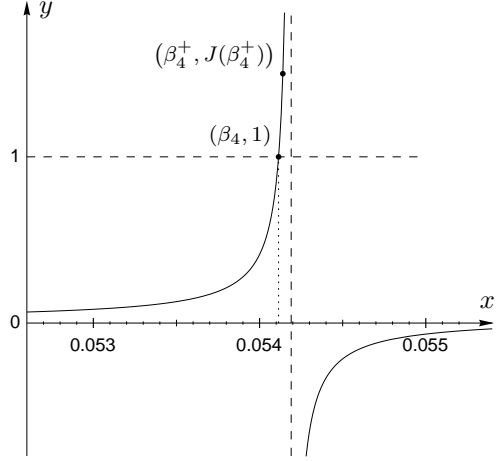}
		\put(-113,120){\footnotesize${{(\beta_4,1)}}$}
		\put(-134,150){\footnotesize${\big(\beta_4^+,J(\beta_4^+)\big)}$}
		\put(-11,57){$x$}
		\put(-177,168){$y$}
\vskip-2mm
	\caption {The graph of $J(x)$.}
	\label{fig.J}
\end{figure}

In the space of operators acting on analytic functions in a neighborhood of $\T^2$,
let us define the following subspaces $\Fr$, $\Fr_1$, $\Fr_2$:
\begin{equation*}
\begin{split}
    T  \in\Fr   \quad&\Leftrightarrow\quad (    T  g)(t,s) = \int_{\T^2}  K(t,s;u,v)\,g(u,v)\,du\,dv, \\
\wh T  \in\Fr_1 \quad&\Leftrightarrow\quad (\wh T  g)(t,s) = \int_{\T}\wh K(t,s;u)\,g(u,s)\,du, \\
\wh T  \in\Fr_2 \quad&\Leftrightarrow\quad (\wh T  g)(t,s) = \int_{\T}\wh K(t,s;v)\,g(t,v)\,dv
\end{split}
\end{equation*}
for some analytic functions $K$ and $\wh K$.
We also set
$$
   \Fr'_j=\{\sigma\circ\wh T\mid\wh T\in\Fr_j\}\qquad(j=1,2),\qquad \text{where $(\sigma g)(t,s) = g(s,t)$.}
$$
By Fredholm's theory, each operator $T\in\Fr$ is compact and, if $I-T$ is invertible, then
$(I-T)^{-1}=I+S$ with $S\in\Fr$.

If $\wh T\in \Fr_1$, then $\wh T$ is no longer compact. However, $\wh T$
can be considered as an analytic family of classical Fredholm operators $\{\wh T_s\}_{s\in\T}$ acting on functions of $t$.
Thus, if $I-\wh T_s$ is invertible for each $s\in\T$, then $(I-\wh T_s)^{-1}$ is an operator of the same form
which analytically depends on $s$ (see, e.g., \cite[Ch.~VI]{refRS}, \cite[\S5.4]{refOre2022}) and hence
$I-\wh T$ is invertible and
\begin{equation}                                 \label{eq.T.inv}
      (I-\wh T)^{-1} = I+\wh S, \qquad \wh S\in\Fr_1.
\end{equation}
It is clear that $\Fr$, $\Fr_j$ are closed by composition and,
for $T\in\Fr$, $\wh T_j\in\Fr_j$, we have
\begin{equation}        \label{eq.TT}
   T\wh T_j,\;\wh T_j T,\;\wh T_j\wh T_{3-j}\in\Fr, \qquad \sigma T_j\sigma \in \Fr_{3-j}.
\end{equation}

Eliminating $h_3$ from \eqref{t4.g2} and \eqref{t4.h3}, we obtain a system
of two equations for indeterminate functions $g_1, g_2$ on $\T^2$. By \eqref{eq.TT} it is of the form
\begin{align}
     g_1 &= \wh T_{11}g_1 + (\wh T_{12} + T_{12})g_2 + \phi,                                \label{eq.fred.1} \\
     g_2 &= (\wh T_{21} + \wh T'_{21} + T_{21})g_1 + T_{22}\,g_2,                           \label{eq.fred.2}
\end{align}
where $\phi=1/(x^2p\Psi_1)$, $T_{ij}\in\Fr$, and
$\wh T_{11}, \wh T_{12}, \wh T_{21}, \wh T'_{21}$ belong to $\Fr_2,\Fr_1,\Fr_1,\Fr'_1$ respectively.
For the integration kernels we shall use the same notation as for the operators but with $T$ replaced by $K$.

One can check numerically that $\|K_{22}\|_2<0.06$ for $x<x_0^+$,
hence the operator $I-T_{22}$ is invertible and the inverse is $I+\wt T_{22}$ with $\wt T_{22}\in\Fr$.
The integration kernel $\wt K_{22}$ can be computed at any point with any precision.
Resolving the equation \eqref{eq.fred.2} with respect to $g_2$
and plugging the result into \eqref{eq.fred.1} we obtain an equation of the form
\begin{equation}                                       \label{eq.fred}
             (I - \wh T_1 - \wh T_2 - T) g_1 = \phi,
\end{equation}
where $\wh T_1=\wh T_{12}\wh T_{21}\in\Fr_1$, $\wh T_2 = \wh T_{11}\in\Fr_2$, and $T\in\Fr$ by \eqref{eq.TT}.
A computation shows that
$\|\wh K_1|_{\T\times\{t\}}\|_2\le 0.83$ and
$\|\wh K_2|_{\{t\}\times\T}\|_2\le 0.53$ for any $x\in[0,x_0^+]$ and any fixed $t\in\T$.
Therefore (see \eqref{eq.T.inv}) the operators $\wh T_j$ are invertible and by \eqref{eq.TT} we have
$$
    (I-\wh T_1)^{-1}(I-\wh T_2)^{-1}(I-\wh T_1 - \wh T_2 - T) = I + \wt T,\qquad \wt T\in\Fr.
$$
A computation shows that $\|\wt T^{64}\|_2\le 0.7$ (we successively computed $\wt T^2$, $\wt T^4$, $\wt T^8,\dots\,$),
hence the operator in \eqref{eq.fred} is invertible and the solution of \eqref{eq.fred} (and hence of the initial system)
analytically depends on $x$. Thus it coincides with $g_1$ for each $x\in[0,x_0^+]$.

If we replace everywhere the integrals by the $n$-th Riemann sums, then the integral equation at each step transforms
to a system of usual linear equations for the values of the involved functions at points all whose coordinates
are $n$-th roots of unity (see \cite[\S5.2]{refOre2022} for estimates of the approximation error).
The discretization of \eqref{eq.fred} is
equivalent to the discretization of the system \eqref{t4.g1}--\eqref{t4.h3}.
In our computation of $c_4$ we used the system \eqref{t4.g1}--\eqref{t4.h3} because the matrix
of its discretization is sparse (it has only $O(n^3)$ non-zero entries) but the sparseness is lost after
the elimination of $h_3$ (see \S\ref{sect.soft} for the programming details).


\section{ Strips of width 5 }\label{sect.5}

In this section we express $c_5$ via solution of a certain system of integral equations
for the generating functions, which allows us to compute $c_5$ with
15 decimal digits. The computation and proofs are similar to those in \S\ref{sect.4}, therefore we
expose them with less details.


\subsection{ Reduction to trapezoids }\label{sect.trapez5}

Let $\lambda,\mu\in\{1,2\}$ and let $a_0,a_5$ be positive integers such that
$a_0-\lambda\equiv a_5\pm\mu\mod 5$. We set $\ell^*_{\lambda\mu}(a_0,a_5)$ to be the number
of primitive triangulations of the trapezoid $T_{5,\lambda}(a_0,a_5)$ spanned by
$(0,0), (\lambda,5), (\lambda+a_5,5), (a_0,0)$.
Let $\ell_{\lambda\mu}(a_0,a_5)$ be the number of its primitive triangulations
without interior side-to-side edges (i.e. edges of the form $[(k_0,0),(k_5,5)]$).
We set $\ell^*_{\lambda\mu}(0,0)=1$, $\ell_{\lambda\mu}(0,0)=0$,
and $\ell^*_{\lambda\mu}(a_0,a_5)=\ell_{\lambda\mu}(a_0,a_5)=0$ when $a_0-\lambda\not\equiv a_5\pm\mu\mod 5$.
Consider the generating functions (see Figure~\ref{figTrapez})
\begin{align*}
  L  _{\lambda\mu}(x)&=\sum_{n\ge 0} \ell  _{\lambda\mu}(n)x^n = \sum_{a_0,a_5\ge 0}\ell  _{\lambda\mu}(a_0,a_5)x^{a_0+a_5},\\
  L^*_{\lambda\mu}(x)&=\sum_{n\ge 0} \ell^*_{\lambda\mu}(n)x^n = \sum_{a_0,a_5\ge 0}\ell^*_{\lambda\mu}(a_0,a_5)x^{a_0+a_5}.
\end{align*}
\if01{ 
$$
  L  _{\lambda\mu}(x) = \sum_{a_0,a_5\ge 0}\ell  _{\lambda\mu}(a_0,a_5)x^{a_0+a_5},\qquad
  L^*_{\lambda\mu}(x) = \sum_{a_0,a_5\ge 0}\ell^*_{\lambda\mu}(a_0,a_5)x^{a_0+a_5}.
$$
}\fi 
\begin{figure} 
	\centering\includegraphics[width=115mm]{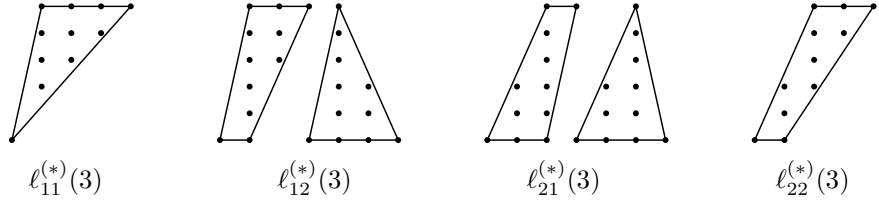} 
	\vskip 5pt
	\centering{$\ell_{11}^{(*)}(3)$} \hskip 22mm
	\centering{$\ell_{12}^{(*)}(3)$} \hskip 22mm
	\centering{$\ell_{21}^{(*)}(3)$} \hskip 22mm
	\centering{$\ell_{22}^{(*)}(3)$} \hskip 10mm
	\caption{Trapezoids that contribute to $\ell^*_{\lambda\mu}(3)$ and $\ell_{\lambda\mu}(3)$.}
	\label{figTrapez}
\end{figure}
Then the matrices
$$
   {\bL}(x)  =\left(\begin{matrix} L_  {11}(x) & L  _{12}(x) \\ L  _{21}(x) & L  _{22}(x) \end{matrix}\right), \qquad
   {\bL}^*(x)=\left(\begin{matrix} L^*_{11}(x) & L^*_{12}(x) \\ L^*_{21}(x) & L^*_{22}(x) \end{matrix}\right)
$$
are symmetric and satisfy the relation (here $\bI$ is the identity matrix)
$$
   \bL^* = \bI + \bL + \bL^2 + \bL^3 + \dots = (\bI - \bL)^{-1}.
$$
We have
$$
   \quad\;\;\bL=\left(\begin{matrix}
     \phantom{x+{}} 2x^2 + 79x^3 + 1075x^4  + \dots &\quad  x + 5x^2 +\phantom{1}84x^3 + 2104x^4 + \dots \\
              x +   5x^2 + 84x^3 + 2104x^4  + \dots &\quad  x + 8x^2 +          111x^3 + 3419x^4 + \dots
   \end{matrix}\right),
$$
$$
   \bL^*=\bI+\left(\begin{matrix}
      \phantom{x+{}} 3x^2 +\phantom{1}90x^3 + 1296x^4 + \dots  &\quad  x +\phantom{1}6x^2 + 101x^3 + 2469x^4 + \dots \\
               x +   6x^2 +          101x^3 + 2469x^4 + \dots  &\quad  x +          10x^2 + 140x^3 + 3965x^4 + \dots
   \end{matrix}\right).
$$
As in \cite{refOre2022} and in \S\ref{sect.4}, for each $\lambda,\mu$ we have
$\lim_n f(5,n)^{1/n} = \lim_n \ell^*_{\lambda\mu}(2n)^{1/n}$, and we deduce
\begin{equation}\label{eq.trapez5}
                                   c_5 = -\frac25\log_2\beta_5,
\end{equation}
where $\beta_5$ is the first real root of the equation $\det(\bI-\bL(x))=0$.

\begin{remark}
The analog of \eqref{eq.trapez4} and \eqref{eq.trapez5} for rectangles of an arbitrary width $m\ge 3$
is $c_m=-\frac2m\log_2\beta_m$ where $\beta_m$ is the first real root of the equation $\det(\bI-\bL(x))=0$
and $\bL(x)$ is a square matrix of size $\varphi(m)/2$ defined in a similar way ($\varphi$ is the Euler totient function).
\end{remark}


\begin{figure}
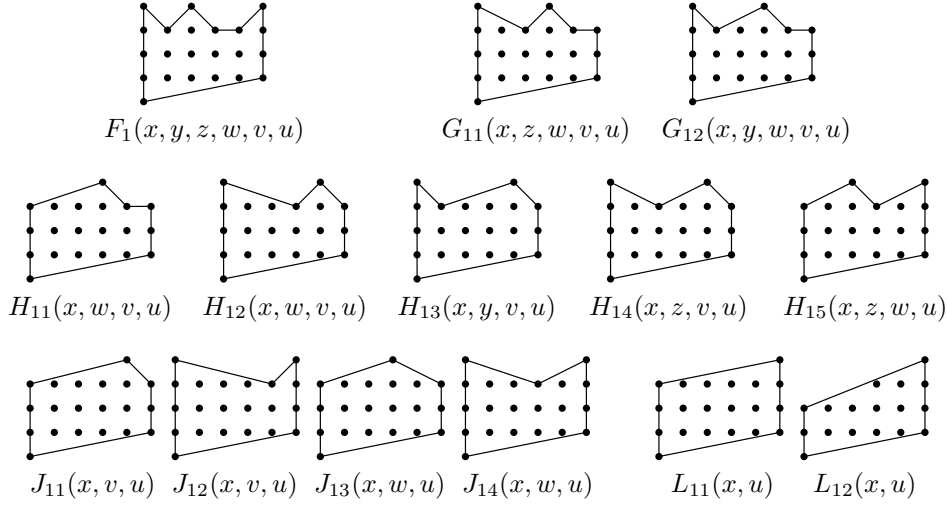
 
	\centering
		\includegraphics[width=90mm]{fg5.eps} 
\centerline{
 $F_1(x,y,z,w,v,u)$  \hskip 17mm
  $G_{11}(x,z,w,v,u)$ \hskip 3mm
  $G_{12}(x,y,w,v,u)$ 
}
\vskip3pt
	\centering
		\includegraphics[width=120mm]{h5.eps} 
\centerline{
  $H_{11}(x,w,v,u)$ \hskip 8pt
  $H_{12}(x,w,v,u)$ \hskip 8pt
  $H_{13}(x,y,v,u)$ \hskip 9pt
  $H_{14}(x,z,v,u)$ \hskip 10pt
  $H_{15}(x,z,w,u)$
}
\vskip3pt
	\centering
		\includegraphics[width=120mm]{jl5.eps} 
\centerline{%
  $J_{11}(x,v,u)$ \hskip 1mm
  $J_{12}(x,v,u)$ \hskip 1mm
  $J_{13}(x,w,u)$ \hskip 1mm
  $J_{14}(x,w,u)$ \hskip 9mm
  $L_{11}(x,u)$ \hskip 4mm
  $L_{12}(x,u)$
}
	\caption{The shapes of $F_1, G_{1\nu}, H_{1\nu}, J_{1\nu}, L_{1\nu}$.}
	\label{figFGHJL}
\vskip-2mm
\end{figure}

\begin{figure}
	\centering
		\includegraphics[width=100mm]{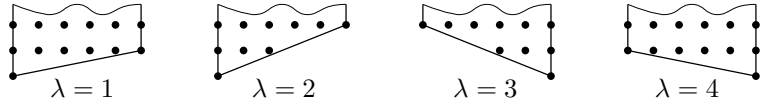}
\vskip-2mm
\centerline{
 $\lambda=1$  \hskip 17mm
 $\lambda=2$  \hskip 17mm
 $\lambda=3$  \hskip 17mm
 $\lambda=4$
}
	\caption{The bottom of the shapes of $F_\lambda, G_{\lambda\nu}, H_{\lambda\nu}, J_{\lambda\nu}, L_{\lambda\nu}$.}
	\label{fig.bottom}
\end{figure}

\subsection{ Recurrence relations }\label{sect.rec5}
We define the generating functions $F_1, G_{11}, G_{12},\dots$  as in \S\ref{sect.rec4}
but according to Figure~\ref{figFGHJL}.
By replacing the bottom parts of the corresponding shapes as in Figure~\ref{fig.bottom} we define
$F_\lambda, G_{\lambda1}, G_{\lambda2},\dots$ (with the same arguments) for $\lambda=1,2,3,4$.
The symmetrizations $\wt F_\lambda, \wt G_{\lambda1}, \wt G_{\lambda2},\dots$ ($\lambda=1,2$)
can be equivalently defined by setting
$$
   \wt F_\lambda = F_\lambda + F_{5-\lambda},\quad
   \wt X_{\lambda,\nu} = X_{\lambda,\nu} + X_{5-\lambda,\nu},
   \qquad X=G,H,J,L,\quad \lambda = 1,2,
$$
with the same arguments. Notice that
\begin{equation}     \label{eq.LL}
    L_{\lambda\mu}(x) = x^{2-\lambda}\wt L_{\lambda\mu}(x,x)
\end{equation}
(see \S\ref{sect.trapez5} for the definition of $L_{\lambda\mu}(x)$).

Since most of the computations do not depend on $\lambda$, we
omit the first subscript. In a few cases when they do depend on $\lambda$, we use
the Kronecker symbol $\delta_{\lambda\mu}$.
In this notation, the recurrence relations from \cite[\S2.1]{refOre2022}
(cf.~\S\ref{sect.rec4}) take the following form:

\begin{align*}
\wt F(x,\dots,u)Q(x,\dots,u)
     & = y^{1/2}\wt G_1\big(xy^{1/2},y^{1/2}z,w,v,u\big)P_1(w,v,u)\\ 
     & + z^{1/2}\wt G_2\big(x,\,yz^{1/2},z^{1/2}w,\,v,u\big)P_2(x,v,u)\\ 
     & + w^{1/2}\wt G_2\big(u,vw^{1/2},w^{1/2}z,y,x\big)P_2(u,y,x)\\ 
     & + v^{1/2}\wt G_1\big(uv^{1/2},v^{1/2}w,\,z,\,y,\,x\big)P_1(z,y,x)\\ 
     & - y^{1/2}w^{1/2}\wt H_4\big(xy^{1/2},y^{1/2}zw^{1/2},w^{1/2}v,u)(1-u)\\
     & - y^{1/2}v^{1/2}\wt H_5\big(xy^{1/2},y^{1/2}z,wv^{1/2},v^{1/2}u)\\
     & - z^{1/2}v^{1/2}\wt H_4\big(uv^{1/2},v^{1/2}wz^{1/2},z^{1/2}y,x)(1-x),
\end{align*}
where
\begin{align*}
  Q(x,y,z,w,v,u)&=1 - x - y - z - w - v - u\\
                &\quad + xz + xw + xv + xu + yw + yv + yu + zv + zu + wu\\
                &\quad - xzv - xzu - xwu - ywu,
\end{align*}
$$
  P_1(w,v,u) = 1 - w - v - u + wu,\qquad
  P_2(x,v,u) = (1-x)(1-v-u),
$$
\begin{align*}
\wt G_1(x,z,w,v,u)P_1(w,v,u)  
     &      = \coef_{\xi^{-1}}\wt F\Big(\frac x\xi,\xi^2,\frac z\xi,w,v,u\Big)P_1(w,v,u) \\ 
     &\quad + z^{1/3}\wt H_1\big(xz^{1/3},z^{2/3}w,v,u\big)(1-u-v)\\
     &\quad + w^{1/2}\wt H_4\big(x,zw^{1/2},w^{1/2}v,u\big)(1-u)\\
     &\quad + v^{1/2}\wt H_5\big(x,z,wv^{1/2},v^{1/2}u\big)\\
     &\quad - w^{1/2}\coef_{\xi^{-1}}\wt G_2\Big(u,vw^{1/2},w^{1/2}\frac z\xi,\xi^2,\frac x\xi\Big)(1-u)\\
     &\quad - v^{1/2}\coef_{\xi^{-1}}\wt G_1\Big(uv^{1/2},v^{1/2}w,\frac z\xi,\xi^2,\frac x\xi\Big)\\
     &\quad - z^{1/3}v^{1/2}\wt J_3\big(xz^{1/3},z^{2/3}wv^{1/2},v^{1/2}u\big),
\end{align*}
\begin{align*}
\wt G_2(x,y,w,v,u)P_2(x,v,u)
     &      = \coef_{\xi^{-1}}\wt F\Big(x,\frac y\xi,\xi^2,\frac w\xi,v,u\Big)P_2(x,v,u)\\
     &\quad + y^{1/3}\wt H_2\big(xy^{2/3},y^{1/3}w,v,u\big)(1-v-u)\\
     &\quad + w^{1/3}\wt H_3\big(x,yw^{1/3},w^{2/3}v,u\big)(1-x)(1-u)\\
     &\quad + v^{1/2}\wt H_4\big(uv^{1/2},v^{1/2}w,y,x\big)(1-x)\\
     &\quad - y^{1/3} v^{1/2}\wt J_4\big(xy^{2/3},y^{1/3}wv^{1/2},v^{1/2}u\big)\\
     &\quad - v^{1/2}\coef_{\xi^{-1}}\wt G_1\Big(uv^{1/2},v^{1/2}\frac w\xi,\xi^2,\frac y\xi,x\Big)(1-x),
\end{align*}

\begin{align*}
\wt H_1(x,w,v,u)(1-v-u)
     &      = \coef_{\xi^{-1}}\wt G_2\Big(\frac x{\xi^2},\xi^3,\frac w\xi,v,u\Big)(1-v-u)\\
     &\quad + w^{1/4}\wt J_1\big(xw^{1/4},w^{3/4}v,u\big)(1-u)\\
     &\quad + v^{1/2}\wt J_3\big(x,wv^{1/2},v^{1/2}u\big)\\
     &\quad - v^{1/2}\coef_{\xi^{-1}}\wt H_4\Big(uv^{1/2},v^{1/2}\frac w\xi,\xi^3,\frac x{\xi^2}\Big),
\end{align*}
\begin{align*}
\wt H_2(x,w,v,u)(1-v-u)
     &      = \coef_{\xi^{-1}}\wt G_1\Big(\frac x\xi,\xi^3,\frac w{\xi^2},v,u\Big)(1-v-u)\\
     &\quad + v^{1/2}\wt J_4\big(x,wv^{1/2},v^{1/2}u\big)\\
     &\quad - v^{1/2}\coef_{\xi^{-1}}\wt H_5\Big(\frac x\xi,\xi^3,\frac w{\xi^2}v^{1/2},v^{1/2}u\Big)	,
\end{align*}
\begin{align*}
\wt H_3(x,y,v,u)
     &      = \coef_{\xi^{-1}}\wt G_2\Big(u,\frac v\xi,\xi^3,\frac y{\xi^2},x\Big)
            + \frac{v^{1/4}}{1-u}\wt J_2\big(uv^{3/4},v^{1/4}y,x\big),
\end{align*}
\begin{align*}
\wt H_4(x,z,v,u)
     &      = \coef_{\xi^{-1}}\wt G_1\Big(x,\frac z\xi,\xi^2,\frac v\xi,u\Big)
       \; +\; \coef_{\xi^{-1}}\wt G_2\Big(u,v,\frac z\xi,\xi^2,\frac x\xi\Big)\\
     &\quad + \frac{v^{1/3}}{1-u}\left\{\wt J_4\big(uv^{2/3},v^{1/3}z,x\big)
            \; -\;   \coef_{\xi^{-1}}\wt H_2\Big(uv^{2/3},v^{1/3}\frac z\xi,\xi^2,\frac x\xi\Big)\right\}\\
     &\quad - \coef_{\xi_1^{-1}\xi_2^{-1}}\wt F\Big(\frac x{\xi_1},\xi_1^2,\frac z{\xi_1\xi_2},\xi_2^2,\frac v{\xi_2},u\Big),
\end{align*}
\begin{align*}
\wt H_5(x,z,w,u)
           = \coef_{\xi^{-1}}\wt G_1\Big(x,z,\frac w\xi,\xi^2,\frac u\xi\Big)
       \, &+\, \coef_{\xi^{-1}}\wt G_1\Big(u,w,\frac z\xi,\xi^2,\frac x\xi\Big)\\
      +  z^{1/3}\wt J_3\big(xz^{1/3},z^{2/3}w,u\big)
  \, &+\,w^{1/3}\wt J_3\big(uw^{1/3},w^{2/3}z,x\big)\\
      -  z^{1/3}\coef_{\xi^{-1}}\wt H_1\Big(xz^{1/3},z^{2/3}\frac w\xi,\xi^2,\frac u\xi\Big)
  \, &-\,w^{1/3}\coef_{\xi^{-1}}\wt H_1\Big(uw^{1/3},w^{2/3}\frac z\xi,\xi^2,\frac x\xi\Big)\\
      &-\, \coef_{\xi_1^{-1}\xi_2^{-1}}\wt F\Big(\frac x{\xi_1},\xi_1^2,\frac z{\xi_1},\frac w{\xi_2},\xi_2^2,\frac u{\xi_2}\Big),
\end{align*}
\begin{align*}
\wt J_1(x,v,u) &= \coef_{\xi^{-1}}\wt H_3\Big(\frac x{\xi^3},\xi^4,\frac v\xi,u\Big) + \frac{\delta_{\lambda,1}}{x(1-u)},\\
\wt J_2(x,v,u) &= \coef_{\xi^{-1}}\wt H_2\Big(\frac x\xi,\xi^4,\frac v{\xi^3},u\Big),
\end{align*}
\begin{align*}
\wt J_3(x,w,u) = \coef_{\xi^{-1}}\wt H_1\Big(x,\frac w\xi,\xi^2,\frac u\xi\Big)
               & + \coef_{\xi^{-1}}\wt H_4\Big(u,\frac w\xi,\xi^3,\frac x{\xi^2}\Big)\\
               &\qquad - \coef_{\xi_1^{-1}\xi_2^{-1}}\wt G_2\Big(\frac x{\xi_1^2},\xi_1^3,\frac w{\xi_1\xi_2},\xi_2^2,\frac u{\xi_2}\Big),\\
\wt J_4(x,w,u)  = \coef_{\xi^{-1}}\wt H_2\Big(x,\frac w\xi,\xi^2,\frac u\xi\Big)
              & + \coef_{\xi^{-1}}\wt H_5\Big(u,\frac w{\xi^2},\xi^3,\frac x\xi\Big)\\
              &\qquad - \coef_{\xi_1^{-1}\xi_2^{-1}}\wt G_1\Big(\frac x{\xi_1},\xi_1^3,\frac w{\xi_1^2\xi_2},\xi_2^2,\frac u{\xi_2}\Big)
                + \frac{x\delta_{\lambda,2}}u,
\end{align*}
$$
\wt L_1(x,u) = \coef_{\xi^{-1}}\wt J_2\Big(\frac u\xi,\xi^5,\frac x{\xi^4}\Big),
\qquad\qquad
\wt L_2(x,u) = \coef_{\xi^{-1}}\wt J_3\Big(\frac x{\xi^2},\xi^5,\frac u{\xi^3}\Big).
$$


\subsection{ A change of variables }\label{sect.ch.var5}
We set
\begin{align*}
  f(t,s,u,v) &= \wt F\Big(\frac{x}{t},\frac{x^2t^2}{s},\frac{x^2s^2}{tu},\frac{x^2u^2}{sv},\frac{x^2v^2}{u},\frac{x}{v}\Big),\\
  g_1(s,u,v) &= \frac{1}{s^{1/2}}\,
    \wt G_1\Big(\frac{x^2}{s^{1/2}},\frac{x^3 s^{3/2}}{u},\frac{x^2 u^2}{sv},\frac{x^2v^2}{u},\frac x v\Big),\\
  g_2(t,u,v) &= \frac{1}{(tu)^{1/2}}\,
    \wt G_2\Big(\frac{x}t,\frac{x^3t^{3/2}}{u^{1/2}},\frac{x^3u^{3/2}}{t^{1/2}v},\frac{x^2v^2}u,\frac x v\Big),\\
  h_k(u,v) &= \frac{1}{u^{k/3}}\wt H_k\Big(\frac{x^3}{u^{1/3}}, \frac{x^4 u^{4/3}}v, \frac{x^2v^2}{u},\frac x v\Big),\quad\qquad k=1,2,\\
  h_3(t,v) &= \frac{1}{t^{2/3}v^{1/3}}\wt H_3\Big(\frac{x}t, \frac{x^4t^{4/3}}{v^{1/3}}, \frac{x^4v^{4/3}}{t^{1/3}}, \frac{x}v\Big),\\
  h_4(s,v) &= \frac{1}{sv^{1/2}}\wt H_4\Big(\frac{x^2}{s^{1/2}}, \frac{x^4 s}{v^{1/2}}, \frac{x^3v^{3/2}}{s^{1/2}}, \frac x v\Big),\\
  h_5(s,u) &= \frac{1}{(su)^{1/2}}\wt H_5\Big(\frac{x^2}{s^{1/2}}, \frac{x^3 s^{3/2}}u, \frac{x^3u^{3/2}}s, \frac{x^2}{u^{1/2}}\Big),\\
  j_1(v)&=\frac 1{v^{1/4}}\wt J_1\Big(\frac{x^4}{v^{1/4}},x^5 v^{5/4},\frac x v\Big),\qquad
  j_3(u) =\frac 1{v^{5/6}}\wt J_3\Big(\frac{x^3}{v^{1/3}},x^5 u^{5/6},\frac{x^2}{u^{1/2}}\Big),\\
  j_2(v)&=\frac 1{v^{3/4}}\wt J_2\Big(\frac{x^4}{v^{1/4}},x^5 v^{5/4},\frac x v\Big),\qquad
  j_4(u) =\frac 1{v^{1/6}}\wt J_4\Big(\frac{x^3}{v^{1/3}},x^5 u^{5/6},\frac{x^2}{u^{1/2}}\Big),\\
  l_1(x)&=\wt L_1\big(x^5,x^5\big),\hskip84pt l_2(x)=\wt L_2\big(x^5,x^5\big).
\end{align*}
We also set
\begin{align*}
  q(t,s,u,v) &= Q\Big(\frac{x}{t},\frac{x^2t^2}{s},\frac{x^2s^2}{tu},\frac{x^2u^2}{sv},\frac{x^2v^2}{u},\frac{x}{v}\Big),\\
  p_1(s,u,v) &= P_1\Big(\frac{x^2u^2}{sv},\frac{x^2v^2}{u},\frac{x}{v}\Big),\qquad
  p_2(t,u,v)  = P_2\Big(\frac{x}{t},\frac{x^2v^2}{u},\frac{x}{v}\Big),\\
  r_1(u,v) &= 1 - \frac{x^2v^2}{u} - \frac{x}{v}, \qquad\qquad\quad\;\; r_2(t,v)=\Big(1-\frac{x}{t}\Big)\Big(1-\frac{x}{v}\Big).
\end{align*}
Then the recurrence relations from \S\ref{sect.rec5} take the following form (cf.~\S\ref{sect.ch.var4}):
\begin{equation}\label{t5.f}
\begin{split}
f(t,s,u,v)q(t,s,u,v)& = xt\,p_1(s,u,v) g_1(s,u,v) + xs\,p_2(t,u,v) g_2(t,u,v) \\
                    & + xv\,p_1(u,s,t)\,g_1(u,s,t) + xu\,p_2(v,s,t)\,g_2(v,s,t)\\
                    & - x^2tu\Big(1-\frac x v\Big) h_4(s,v)
                      - x^2sv\Big(1-\frac x t\Big) h_4(u,t)
                      - x^2tv\,h_5(s,u), 
\end{split}
\end{equation}
\begin{equation}\label{t5.g1}
\begin{split}
g_1(s,u,v)p_1(s,u,v)& = \frac x s\,p_1(s,u,v)\coef_{t^{-1}} f(t,s,u,v)\\
                    &\quad + x\,r_1(u,v)h_1(u,v) + xu\Big(1-\frac x v\Big)h_4(s,v) + xv\,h_5(s,u)\\
                    &\quad - \frac{x^2u}s\Big(1-\frac x v\Big)\coef_{t^{-1}} g_2(v,s,t)
                      - \frac{x^2v}s\coef_{t^{-1}} g_1(u,s,t)
                      - x^2v\, j_3(u),
\end{split}
\end{equation}
\begin{equation}\label{t5.g2}
\begin{split}
g_2(t,u,v)p_2(t,u,v) &= \frac{x}{tu}\,p_2(t,u,v)\coef_{s^{-1}} f(t,s,u,v)\\
                     &\quad + x\,r_1(u,v)h_2(u,v) + x\,r_2(t,v)h_3(t,v) + xv\Big(1-\frac{x}{t}\Big)h_4(u,t)\\
                     &\quad - \frac{x^2v}u\,j_4(u) - \frac{x^2v}{ut}\Big(1-\frac{x}{t}\Big)\coef_{s^{-1}}g_1(u,s,t),
\end{split}
\end{equation}
\begin{align}
\begin{split}\label{t5.h1}
h_1(u,v)r_1(u,v) &= x\,r_1(u,v)\coef_{t^{-1}} g_2(t,u,v) + x\Big(1-\frac x v\Big)j_1(v) + xv\,j_3(u) \\
                 &\qquad- x^2v\coef_{t^{-1}}h_4(u,t),
\end{split}\\
h_2(u,v)r_1(u,v) &= \frac x u\,r_1(u,v)\coef_{s^{-1}}g_1(s,u,v) + \frac{xv}u j_4(u) - \frac{x^2v}u\coef_{s^{-1}}h_5(s,u),\label{t5.h2}
\end{align}
\begin{align}
h_3(t,v) &= \frac{x}t\coef_{s^{-1}} g_2(v,s,t) + \frac{vx}{v-x}\,j_2(t),\label{t5.h3}\\
\begin{split}\label{t5.h4}
h_4(s,v) &= \frac x{sv}\coef_{u^{-1}} g_1(s,u,v) + \frac x s \coef_{t^{-1}} g_2(v,s,t) + \frac{xv}{s(v-x)}j_4(s)\\
         &\quad - \frac{x^2}{s^2v}\coef_{t^{-1}u^{-1}}f(t,s,u,v) - \frac{x^2v}{s(v-x)}\coef_{t^{-1}}h_2(s,t),
\end{split}\\
\begin{split}\label{t5.h5}
h_5(s,u) &= \frac{x}u\coef_{v^{-1}} g_1(s,u,v) + \frac{x}s\coef_{t^{-1}} g_1(u,s,t) + x\,j_3(u) + x\,j_3(s)\\
         &\quad
            - \frac{x^2}u\coef_{v^{-1}} h_1(u,v) - \frac{x^2}s\coef_{t^{-1}} h_1(s,t)
            - \frac{x^2}{su}\coef_{t^{-1}v^{-1}}f(t,s,u,v),
\end{split}
\end{align}
\begin{align}
j_1(v) &= x\coef_{t^{-1}}h_3(t,v) + \frac{v\delta_{\lambda,1}}{x^4(v-x)},\label{t5.j1}\\
j_2(v) &= \frac x v\coef_{u^{-1}}h_2(u,v),\label{t5.j2}\\
j_3(u) &= \frac x u\coef_{v^{-1}} h_1(u,v) + x\coef_{t^{-1}}h_4(u,t) - \frac{x^2}u\coef_{t^{-1}v^{-1}}g_2(t,u,v),\label{t5.j3}\\
j_4(u) &=       x  \coef_{v^{-1}} h_2(u,v) + x\coef_{s^{-1}}h_5(u,s) - \frac{x^2}u\coef_{s^{-1}v^{-1}}g_1(s,u,v) + x\delta_{\lambda,2}\label{t5.j4},
\end{align}
\begin{equation}\label{t5.l12}
l_1(x) = x\coef_{v^{-1}}j_2(v),\qquad\qquad l_2(x) = x\coef_{u^{-1}}j_3(u).
\end{equation}


\subsection{ Elimination of $f$, $h_1$, $h_2$, $j_1$, $j_2$, $j_4$ }\label{sect.elim5}
Let
\begin{xalignat*}{2}
   &\Phi_1(s,u,v)  = \coef_{t^{-1}}\frac t{q(t,s,u,v)},        &&\Psi_1(s,u,v) = 1-\frac{x^2\!}s\,p_1(s,u,v)\Phi_1(s,u,v),\\
   &\Phi_2(t,u,v)  = \coef_{s^{-1}}\frac s{q(t,s,u,v)},        &&\Psi_2(t,u,v) = 1-\frac{x^2}{tu}\,p_2(t,u,v)\Phi_2(t,u,v),\\
   &\Phi_{13}(s,v) = \coef_{t^{-1}u^{-1}}\frac{tu}{q(t,s,u,v)},&&\Psi_{13}(s,v)= 1-\frac{x^4(v-x)}{s^2 v^2}\Phi_{13}(s,v),\\
   &\Phi_{14}(s,u) = \coef_{t^{-1}v^{-1}}\frac{tv}{q(t,s,u,v)},&&\Psi_{14}(s,u)= 1-\frac{x^4}{su}\Phi_{14}(s,u),\\
   &\Phi_0(u)      = \coef_{v^{-1}}\frac{v}{r_1(u,v)},         &&\Psi_0(u)     = 1-\frac{x^2}{u}\Phi_0(u).
\end{xalignat*}

\begin{remark}\label{rem.phi0} We have (see OEIS \cite{refOEIS}, A025174)
$$
     \frac{x^2}u\Phi_0(u) = \sum_{n=1}^\infty \binom{3n-1}{n-1}\frac{x^{4n}}{u^n}.
$$
\end{remark}

\begin{lemma}\label{lem.phi0} One has
\begin{equation}\label{eq.phi0}
      h_2(u,v) = \frac{x}u\coef_{s^{-1}}g_1(s,u,v) + \frac{x^2 v\delta_{\lambda,2}}{u\,r_1(u,v)\Psi_0(u)}.
\end{equation}
\end{lemma}

\begin{proof}
Eliminating $j_4(u)$ from \eqref{t5.h2} and \eqref{t5.j4}, we obtain the equation
\begin{equation}\label{eq1.proof.phi0}
          \chi(u,v) = \frac{x^2v}{u\,r_1(u,v)}\,\big(\delta_{\lambda,2} + \coef_{w^{-1}}\chi(u,w) \big),
\end{equation}
where
$$
    \chi(u,v) = h_2(u,v) - \frac{x}u\coef_{s^{-1}}g_1(s,u,v).
$$
By equating the coefficients of $v^{-1}$ in both sides of \eqref{eq1.proof.phi0}, we obtain
\begin{equation}\label{eq2.proof.phi0}
        \coef_{v^{-1}}\chi(u,v) = \frac{x^2}u\Phi_0(u)\big(\delta_{\lambda,2} + \coef_{w^{-1}}\chi(u,w) \big).
\end{equation}
Noting that $\coef_{v^{-1}}\chi(u,v)=\coef_{w^{-1}}\chi(u,w)$, we find $\coef_{w^{-1}}\chi(u,w)$ from 
\eqref{eq2.proof.phi0} and plug it into \eqref{eq1.proof.phi0}, which yields \eqref{eq.phi0} after simplifications.
The lemma is proven.
\end{proof}

\begin{lemma}\label{lem.L2} One has
\begin{equation}\label{eq.L2}
      l_2(x)  =  x^2 \coef_{t^{-1}u^{-1}} h_4(u,t)  +  x^3\coef_{v^{-1}} j_1(v)
\end{equation}
\end{lemma}

\begin{figure}
	\centering
		\includegraphics[width=80mm]{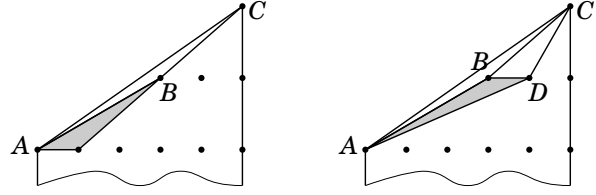}
\vskip-1mm
	\caption {To the proof of Lemma \ref{lem.L2}.}
	\label{fig.lem.L2}
\end{figure}

\begin{proof}
Any primitive triangulation contributing to $l_2(x)$ has the triangle $ABC$ shown in Figure~\ref{fig.lem.L2}.
There are exactly two possibilities for the triangle adjacent to the edge $AB$ from below:
the gray triangles in Figure~\ref{fig.lem.L2}. In the second case we necessarily have also the edge $CD$.
In terms of the generating functions these observations mean \eqref{eq.L2}. The lemma is proven.
\end{proof}

\begin{remark} A {\tt Wolfram Mathematica} code
that checks the identities \eqref{t5.g1-1} -- \eqref{t5.l12-1} below up to $O(x^n)$ is available at
  \url{https://www.math.univ-toulouse.fr/~orevkov/tr45.html}.
\end{remark}

By plugging \eqref{t5.f}, \eqref{t5.h1}, and then \eqref{t5.j1}, into \eqref{t5.g1} we obtain
\begin{equation}\label{t5.g1-1}
\begin{split}
     \Psi_1&(s,u,v) g_1(s,u,v) =
     \frac{\delta_{\lambda,1}}{x^2 p_1(s,u,v)}
     \;+\; \frac{x\Psi_1(s,u,v)}{v\,p_1(s,u,v)}\left\{ u(v-x) h_4(s,v)  \;+\;  v^2 h_5(s,u) \right\}
\\&\qquad
    + \coef_{t^{-1}}\!\left\{ \frac{x^2v}s \left( \frac{p_1(u,s,t)}{q(t,s,u,v)} - \frac{1}{p_1(s,u,v)} \right) g_1(u,s,t)
    \;+\; \frac{x^3(v-x)}{v\,p_1(s,u,v)} h_3(t,v)\right\}
\\&\qquad
    + \frac{x^2u(v-x)}{sv} \coef_{t^{-1}} \left( \frac{r_1(s,t)}{q(t,s,u,v)} - \frac 1{p_1(s,u,v)} \right) g_2(v,s,t)
\\&\qquad
    + \coef_{t^{-1}} \left( \frac{x^2(t-x)}{t\,q(t,s,u,v)} + \frac{x^2}{p_1(s,u,v)} \right)\!\big\{r_1(u,v) g_2(t,u,v)- xv\, h_4(u,t)\big\}.\\
\end{split}
\end{equation}

By plugging \eqref{t5.f} and \eqref{t5.h2} into \eqref{t5.g2} we obtain
\begin{equation}\label{t5.g2-1}
\begin{split}
     \Psi_2&(t,u,v) g_2(t,u,v) =
       \frac{x(v-x)}{v\,r_1(u,v)} h_3(t,v)
      \;+\; \frac {xv\Psi_2(t,u,v)}{r_1(u,v)} h_4(u,t) 
  \\&\quad
         + \coef_{s^{-1}}\left\{ \frac{x^2v}{tu} \left( \frac{p_1(u,s,t)}{q(t,s,u,v)} - \frac 1{r_1(u,v)} \right) g_1(u,s,t)
       \;+\; \frac{x^2p_2(v,s,t)}{t\,q(t,s,u,v)} g_2(v,s,t)\right\}
  \\&\quad
    + \frac{x^2}u \coef_{s^{-1}}\left( \frac{p_1(s,u,v)}{q(t,s,u,v)} + \frac t{t-x}   \right) g_1(s,u,v)
       \;-\; \frac{x^3(v-x)}v\coef_{s^{-1}}\frac{h_4(s,v)}{q(t,s,u,v)}
  \\&\quad
      - \frac{x^3v}u \coef_{s^{-1}} \left( \frac 1{q(t,s,u,v)} + \frac 1{p_2(t,u,v)} \right)h_5(s,u).
\end{split}
\end{equation}
By plugging \eqref{t5.f} and \eqref{t5.j4} (with $t,s,u,v$ replaced by $v,u,s,t$ respectively)
into \eqref{t5.h4} we obtain
\begin{equation}\label{t5.h4-1}
\begin{split}
\Psi_{13}(s,v)& h_4(s,v)
    = \frac{x^2v\,\delta_{\lambda,2}}{s(v-x)}
  \;+\; \frac{x}{sv}\coef_{u^{-1}}\!\Psi_1(s,u,v)g_1(s,u,v)
  \;+\; \frac{x}{s} \coef_{t^{-1}}\!\Psi_2(v,s,t)g_2(v,s,t)\\
  & - \frac{x^3}{s}\coef_{t^{-1}u^{-1}}\!\left\{\left( \frac v{s(v-x)} + \frac{p_1(u,s,t)}{s\,q(t,s,u,v)} \right) g_1(u,s,t)
                  \;+\; \frac{p_2(t,u,v)}{v\,q(t,s,u,v)} g_2(t,u,v)\right\}\\
  & + \coef_{u^{-1}}\! \left(\frac{x^2v}{s(v-x)} + \frac{x^4}{s^2}\Phi_1(s,u,v)\right) h_5(s,u)
    \;+\;  \coef_{t^{-1}u^{-1}}\! \frac{x^4(t-x)}{ts\,q(t,s,u,v)} h_4(u,t).
\end{split}
\end{equation}
By plugging \eqref{t5.f}, \eqref{t5.j3}, and \eqref{t5.j3} with $t,s,u,v$ replaced by $v,u,s,t$ 
into \eqref{t5.h5} we obtain
\begin{equation}\label{t5.h5-1}
\begin{split}
\Psi_{14}(s,u) h_5(s,u)
    =   \frac{x}{u}\coef_{v^{-1}}\Psi_1(s,u,v)g_1(s,u,v)
  &\;+\; \frac{x}{s}\coef_{t^{-1}}\Psi_1(u,s,t)g_1(u,s,t)\\
      + \coef_{v^{-1}}\!\Big(x^2 + \frac{x^4(v-x)}{sv}\Phi_1(s,u,v)\Big)h_4(s,v)
 &\;+\; \coef_{t^{-1}}\!\Big(x^2 + \frac{x^4(t-x)}{tu}\Phi_1(u,s,t)\Big)h_4(u,t)\\
  &\hskip-185pt        
    - \coef_{t^{-1}v^{-1}}\left\{ \frac{x^3}u\Big(1 + \frac{p_2(t,u,v)}{q(t,s,u,v)}\Big){g_2(t,u,v)}
                            \;+\; \frac{x^3}s\Big(1 + \frac{p_2(v,s,t)}{q(t,s,u,v)}\Big){g_2(v,s,t)}\right\}.
\end{split}
\end{equation}
Eliminating $h_2(u,v)$ from \eqref{t5.j2} and \eqref{eq.phi0} we obtain
\begin{equation}\label{t5.j2-1}
 j_2(v)=\coef_{s^{-1}u^{-1}}\frac{x^2}{uv}g_1(s,u,v)\;+\;\coef_{u^{-1}}\frac{x^3\delta_{\lambda,2}}{u\,r_1(u,v)\Psi_0(u)}.
\end{equation}

Then \eqref{t5.h3} and \eqref{t5.g1-1}--\eqref{t5.j2-1} is a system of six equations for
$g_1$, $g_2$, $h_3$, $h_4$, $h_5$, $j_2$.
Our final goals are $l_1$ and $l_2$. The series $l_1$ is already expressed via $j_2$ in \eqref{t5.l12}.
Eliminating $j_1$ from \eqref{eq.L2} and \eqref{t5.j1} we express $l_2$ via $h_3$ and $h_4$. Thus
\begin{equation}\label{t5.l12-1}
               l_1 = x\coef_{v^{-1}}j_2(v),\qquad\quad
               l_2 = x^4\coef_{t^{-1}v^{-1}}\!h_3(t,v)
                     \,+\, x^2\coef_{s^{-1}v^{-1}}\!h_4(s,v) \,+\, \delta_{\lambda,1}.
\end{equation}
Notice also that \eqref{t5.j2-1} combined with \eqref{t5.l12-1} yields
$$
    l_1 = \coef_{(suv)^{-1}}\frac{x^3}{uv} g_1(s,u,v) + x^5\delta_{\lambda,2}.
$$


\subsection{ Computation of $c_5$ }
The computation of $c_5$ and its justification are as in Sections~\ref{sect.converge4} and \ref{sect.fred4}
and we omit the details. In brief,
for a fixed $x\in[0,\beta_5^{1/5}+\varepsilon]$, we replace the ``$\coef\,$'' by Cauchy integrals
in \eqref{t5.h3} and \eqref{t5.g1-1}--\eqref{t5.j2-1}, discretize
the obtained system of integral equations, solve them, plug the solution to \eqref{eq.L2} and \eqref{t5.j1},
and compute $\bL$ by \eqref{eq.LL}. In this way we solve numerically the equation $\det(\bI-\bL(x))=0$ and find $c_5$ by 
\eqref{eq.trapez5}.


\section{Computational issues}\label{sect.compute}

\subsection{ Improving the rate of convergence of Riemann sums }

Given an analytic function $f$ on $\T$,
the rate of convergence of the Riemann sums to the Cauchy integral
$$
    \frac{1}{n}\sum_{k=1}^n f(e^{2\pi i k/n}) \;\longrightarrow\; \int_0^1 f(e^{2\pi i t})\,dt = \frac{1}{2\pi i}\int_\T \frac{f(z)}z dz
$$
depends on the maximal width of an annulus of the form $1/r < [z| < r$ to which $f$ extends analytically;
see \cite[Lemma~5.1]{refOre2022}. 
Thus, if $\mu$ is a linear fractional transformation such that $\mu(\T)=\T$ and the function $f\circ\mu$
extends to a wider annulus of this form, then the Riemann sums converge faster after the change of variable $z=\mu(\zeta)$.
It turns out that the accuracy of computation of the functions $H(x)$ (in \cite[\S4.4]{refOre2022}),
$J(x)$ (in \S\ref{sect.4}), and $L_{\lambda\mu}(x)$ (in \S\ref{sect.5}) becomes almost twice better
after the change of variable
\begin{equation}\label{eq.b}
           z = (\zeta+b)/(b\zeta+1)
\end{equation}           
with $b=1/3$ applied at all steps of
the computation; see the middle column of Table~\tabErr.

\begin{remark}
Surprisingly, the upper bounds of the $L^2$-norms of the integration kernels in \S\ref{sect.fred4}
become worse after this variable change with $b>0$ but better with $b<0$.
\end{remark}

Another simple trick (we call it the {\it \quo-shift\/}),
which allowed us to improve the rate of convergence is based
on the fact that all the series we consider have real coefficients. Namely,
let $f$ be as above and let $\sum_{m\in\Z} c_m z^m$ be its Laurent series.
Since $f$ is analytic in a neighborhood of $\T$, we have $|c_m|=o(r^{|m|})$, $0<r<1$.
Let $\omega = e^{2\pi i/n}$. Then the approximation error for the Riemann sum over the nodes $\omega^k$,
$k=1,\dots,n$, is
\begin{equation*}
\begin{split}
   \int_0^1 f(e^{2\pi i t})\,dt - \frac1n\sum_{k=1}^n f(\omega^k)
   &=  c_0 - \frac1n\sum_{k=1}^n f(\omega^k)
   =  c_0 - \frac1n\sum_{k=1}^n\sum_{m\in\Z}c_m \omega^{km}
\\
   &=  c_0 - \frac1n\sum_{m\in\Z}c_m\sum_{k=1}^n \omega^{km}
   =  c_0 - \!\!\sum_{m\equiv 0(n)}\!\!\! c_m = o(r^n),
\end{split}
\end{equation*}
whereas, if all the $c_m$ are real, then the approximation error for the real part of the Riemann sum
over the shifted nodes $\omega_0\omega^k$, where $\omega_0 = e^{2\pi i/4n}$ (and hence $\omega_0^n=i$), is 
\begin{equation*}
\begin{split}
   \int_0^1 f(e^{2\pi i t})\,dt &- \frac1n\re\sum_{k=1}^n f(\omega_0\omega^k)
   =  c_0 - \frac1n\re\sum_{k=1}^n f(\omega_0\omega^k)
\\
  &=  c_0 - \frac1n\re\sum_{k=1}^n\sum_{m\in\Z}c_m \omega_0^m\omega^{km}
   =  c_0 - \frac1n\re\sum_{m\in\Z}c_m\omega_0^m\sum_{k=1}^n \omega^{km}
\\
   &=-\re(\; \dotsc ic_{-3n} - c_{-2n} - ic_{-n} + ic_n - c_{2n} - ic_{3n} +\dots\;) = o(r^{2n}).
\end{split}
\end{equation*}

We illustrate the efficiency of these improvements in the following example. In \cite[Figure~7]{refOre2022}
is presented a {\tt Wolfram Mathematica} function {\tt H} that computes the $n$-th approximation of $H(x)$ for a given $x$,
where $H$ is the function such that the first real solution $x_0$ of $H(x)=1$ determines $c_3$ by
the formula $c_3=-\frac23\log_2 x_0$.
In \cite[Table~7]{refOre2022}, the results of computation of $H(x_0)-1$ by this program are given for $n=100,200,\dots,1200$.
They are reproduced in the left column of Table~\tabErr\ here.

\begin{figure}
	\centering
		\includegraphics[width=140mm]{prog-tr3.eps}
\vskip-1mm
	\caption {The program from \cite{refOre2022} improved using \eqref{eq.b} and the \quo-shift.}
	\label{fig.Prog}
\end{figure}

In Figure~\ref{fig.Prog} we present a modification of the program from \cite{refOre2022} implementing the above improvements.
We see in Table~\tabErr\ that the impact of both improvements is significant. The program used for the middle column is
as in Figure~\ref{fig.Prog} but without the \quo-shift, i.e., with ``{\tt omega0=}...'' replaced by ``{\tt omega0=1}''
(here we used $x_0$ computed to 1100 digits). 

\vbox{%
 \centerline{Table \tabErr. Approximations of $H(x_0)-1$. }
 \centerline{$b$ is from \eqref{eq.b}; $t_{\tt H}$ and $t_{\tt H2}$ are the CPU time for {\tt H} and {\tt H2}. }
 \medskip
 \centerline{
   \vbox{\offinterlineskip
   \def\h {height2pt&\omit&&\omit&&\omit&&\omit&&\omit&\cr}
   \def\o{\omit} \def\t{\times 10} \def\s{\;\;} \def\ss{\s\s} \def\b{\!\!}
   \def\q{\quad\;}
   \hrule
   \halign{&\vrule#&\strut\;\hfil\s#\s\hfil\cr
   \h
   & $n$ && {\tt H} from \cite{refOre2022} && $b=\frac13$, no \quo-shift   && {\tt H2} in Fig.\ref{fig.Prog}  && $t_{\tt H2}/t_{\tt H}$ &\cr
   \h
   \noalign{\hrule}
   \h
   &  100 && $1.44\t^{-10}\s$ && $-4.05\t^{-21}\s$ && $7.96\t^{-30}\s$ && 1.19 &\cr\h  
   &  200 && $5.01\t^{-22}\s$ && $-6.95\t^{-42}\s$ && $3.60\t^{-60}\s$ && 1.22 &\cr\h  
   &  300 && $1.73\t^{-33}\s$ && $-8.63\t^{-63}\s$ && $1.79\t^{-90}\s$ && 1.31 &\cr\h  
   &  400 && $6.02\t^{-45}\s$ && $-9.56\t^{-84}\s$ && $9.35\t^{-121}$  && 1.33 &\cr\h  
   &  500 && $2.09\t^{-56}\s$ && $-9.92\t^{-105}$  && $4.99\t^{-151}$  && 1.37 &\cr\h  
   &  600 && $7.26\t^{-68}\s$ && $-9.88\t^{-126}$  && $2.71\t^{-181}$  && 1.41 &\cr\h  
   &  700 && $2.52\t^{-79}\s$ && $-9.58\t^{-147}$  && $1.49\t^{-211}$  && 1.47 &\cr\h  
   &  800 && $8.78\t^{-91}\s$ && $-9.09\t^{-168}$  && $8.25\t^{-242}$  && 1.57 &\cr\h  
   &  900 && $3.06\t^{-102}$  && $-8.49\t^{-189}$  && $4.61\t^{-272}$  && 1.66 &\cr\h  
   & 1000 && $1.06\t^{-113}$  && $-7.84\t^{-210}$  && $2.59\t^{-302}$  && 2.01 &\cr\h  
   & 1100 && $3.72\t^{-125}$  && $-7.16\t^{-231}$  && $1.45\t^{-332}$  && 2.01 &\cr\h  
   & 1200 && $1.29\t^{-136}$  && $-6.49\t^{-252}$  && $8.27\t^{-363}$  && 2.03 &\cr\h  
   \noalign{\hrule}
   }}
 }
}


\subsection{ The software used for solving big linear systems with high precision } \label{sect.soft}
The systems of linear equations in \cite{refOre2022} were not too big, and we used
{\tt Wolfram Mathematica} to solve them (see \cite[Figure~7]{refOre2022} and Figure~\ref{fig.Prog}).
The size of the linear systems in the present paper exceeds the capacity of {\tt Mathematica}, and we used {\tt MATLAB}
(the code is available at
  \url{https://www.math.univ-toulouse.fr/~orevkov/tr45.html}
  as well as 1100 digits of $c_3$).
For computations in \S\ref{sect.5}, the standard precision was enough because
anyway it was impossible to find $c_5$ with a higher precision since the number of
equations grows too rapidly.

As for computations in \S\ref{sect.4}, it is possible to do them with high precision
but since we did not have access to a high precision version of {\tt MATLAB},
we combined {\tt MATLAB} and {\tt Mathematica}.
Namely,
in order to solve a matrix equation $AX=B$
we computed successive approximations of the solution $X_0,X_1,X_2,\dots$ where
$X_0=0$,
$X_{k+1}=X_k + x_k$, and $x_k$ is a solution of $AX=B_k$ computed by {\tt MATLAB}
where $B_k$ is $B-AX_k$ computed by {\tt Mathematica} with a suitable higher precision.

The computation of $B_k$ can be performed without storing the whole high precision matrix $A$
to the memory, because each entry of $A$ is used only once.
Since the matrix of the system \eqref{t4.g1}--\eqref{t4.h3} is sparse (its dimension is of
order $n^2$ but the number of non-zero entries is of order $n^3$), the computation of
$B_k$ with {\tt Mathematica} takes a reasonable time.


\section{Exact values and empirical estimates }\label{sect.exact}

\subsection{Exact values}
In \cite[\S 2.2]{refOre2022} we reported about 
some computed exact values of $f(m,n)$.
Since then we have modified the C-program that was used for these
computations: optimized the memory allocation
(storing in most cases 32-bit offsets instead of 64-bit pointers) 
and parallelized the computations using the {\tt pthread} library.
This allowed us to compute
$f(6,n)$ for $n=51,\dots,74$,
$f(7,n)$ for $n=21,\dots,38$,
$f(8,n)$ for $n=14,\dots,20$,
$f(9,n)$ for $n=10,11,12$, and
$f(10,10)$. In particular,
$$
  f(10,10) = 
  14961 82791 49336 06361 14702 97684 37248 21429 66337 89595 10906 99398 812
$$
(this was the most extensive computation) and
\begin{align*}
  f(6,74) =
  &\; 11003439126736826379022902097610783003781798654205692030019324180445118\\
  &\; 30333800021120430306017131753558876942975300335453133989087039447873894\\
  &\; 00069135251552234037734742405444996579954965774712177539373757147267542\\
  &\; 4020011731373405247672299081300738745960922639298730311421288165457168,
\end{align*}
which gives $c_{6,74} \approx 2.11018$. This is the largest computed
capacity of a rectangle but it is less than $c_5\approx 2.11801 $.
The C-program and all the computed exact values are available at

\smallskip
\centerline{
  \url{https://www.math.univ-toulouse.fr/~orevkov/tr.html}.
}
\smallskip
\noindent
(some them can be found in OEIS \cite{refOEIS}: A082640, A296165, A351484--A351488).


\subsection{Empirical estimates} \label{sect.empi}

In this subsection we present empirical estimates of $c_m$ via the known exact values
of $f(m,n)$. We use the method proposed in \cite[\S6]{refLZ}.
For a fixed $m$ with known values of $f(m,1),\dots,f(m,2k+2)$, we find
$A$, $a_1,\dots,a_k$, and $b_1,\dots,b_k$ such that
\begin{equation}\label{eq.LZ}
               \frac{f(m,n+1)}{f(m,n)} = A\frac{n^k + a_1 n^{k-1} + \dots + a_k}
                                               {n^k + b_1 n^{k-1} + \dots + b_k}
\end{equation}
for $n=1,\dots,2k+1$,
and assume that the equation \eqref{eq.LZ} is a good approximation of the quotient
$f(m,n+1)/f(m,n)$ for all $n$. If $a_1\ne b_1$, this assumption implies that
(see \cite[\S6]{refLZ})
$$
         f(m,n) \sim \operatorname{const}\cdot A^n\, n^\alpha,\qquad\alpha={a_1-b_1}.
$$

Let $A_m^{(k)}$ and $\alpha_m^{(k)}$ be the constants $A$ and $\alpha$ found in this way
for given $m$ and $k$, and let 
$c_m^{(k)}=\frac{1}{m}\log_2 A_m^{(k)}$.
We have
\begin{equation}\label{eq.emp}
\begin{split}
    |c_3^{(399)}-c_3|\approx 10^{-40},\qquad
   &|c_4^{( 99)}-c_4|\approx 10^{-12},\qquad
    |c_5^{( 56)}-c_5|\approx 10^{ -8},
\\
    |\alpha_3^{(399)}+0.5|\approx 10^{-35},\qquad
   &|\alpha_4^{( 99)}+0.5|\approx 10^{ -8},\qquad
    |\alpha_5^{( 56)}+0.5|\approx 10^{ -5}.
\end{split}
\end{equation}
In Figure~\ref{fig.emp} we show the difference $c_m^{(k)}-c_m$ for $m=3,4,5$.
We see that, for the most of the values of $k$, 
this difference is less than $10^{-2-0.11\,k}$, though there are some rare exceptions,
for example so are $c_4^{(12)}=2.0927$, $c_5^{(13)}=2.1725$, $c_6^{(20)}=2.0714$, and $c_6^{(35)}=2.1288$
(see the dashed lines in Figure~\ref{fig.emp}).
Based on these data, one can expect that $c_6^*=2.1284$ and $c_7^*=2.136$ are good approximations of $c_6$ and $c_7$
(see Figure~\ref{fig.emp} and Tables~\tabEmp\ and \tabEmpCseven). The extrapolation using \eqref{eq.LZ} with
the computed values and empirical estimates of $\lim_n f(m,n)^{1/n}$ for $m=0,\dots,7$ (as well as for
$m=1,\dots,6$ or $m=2,\dots,7$) allows us to expect that $2.2<c<2.3$.

\if01{
A6 = << "~/TR/PADE/dA6.txt"
c6 = 43 * 10^13
c6 = 437263 * 10^15
N[Table[(A6[[n]]-c6)/10^(20-n/10),{n,3,35}]]
}\fi 

\medskip
\vbox{%
 \centerline{Table \tabEmp. Empirical estimates of $c_6$. }
 \smallskip
 \centerline{
   \vbox{\offinterlineskip
   \def\h {height2pt&\omit&&\omit&&\omit&&\omit&&\omit&&\omit&&\omit&&\omit&\cr}
   \def\o{\omit} \def\t{\times 10} \def\s{\;\;} \def\ss{\s\s} \def\b{\!\!}
   \def\q{\quad\;}
   \hrule
   \halign{&\vrule#&\strut\;\hfil\s#\s\hfil\cr
   \h
   & $k$ && $c_6^{(k)}$ && $k$   && $c_6^{(k)}$  && $k$ && $c_6^{(k)}$ && $k$ && $c_6^{(k)}$ &\cr
   \h
   \noalign{\hrule}
   \h
%
   &  5 &&  2.127148  && 13 &&  2.127743  && 21 &&  2.1283724  && 29 &&  2.1283929  &\cr\h
   &  6 &&  2.130178  && 14 &&  2.127968  && 22 &&  2.1284091  && 30 &&  2.1283885  &\cr\h
   &  7 &&  2.128419  && 15 &&  2.127263  && 23 &&  2.1282670  && 31 &&  2.1283897  &\cr\h
   &  8 &&  2.134234  && 16 &&  2.127728  && 24 &&  2.1283909  && 32 &&  2.1283895  &\cr\h
   &  9 &&  2.127842  && 17 &&  2.127728  && 25 &&  2.1283861  && 33 &&  2.1283899  &\cr\h
   & 10 &&  2.127929  && 18 &&  2.127766  && 26 &&  2.1283909  && 34 &&  2.1284081  &\cr\h
   & 11 &&  2.127791  && 19 &&  2.128433  && 27 &&  2.1283921  && 35 &&  2.1287987  &\cr\h
   & 12 &&  2.127888  && 20 &&  2.071426  && 28 &&  2.1283915  && 36 &&  2.1284021  &\cr\h
   \noalign{\hrule}
   }}
 }
}
\medskip
\vbox{%
 \centerline{Table \tabEmpCseven. Empirical estimates of $c_7$. }
 \smallskip
 \centerline{
   \vbox{\offinterlineskip
   \def\h {height2pt&\omit&&\omit&&\omit&&\omit&&\omit&&\omit&&\omit&&\omit&&\omit&&\omit&\cr}
   \def\o{\omit} \def\t{\times 10} \def\s{\;\,} \def\ss{\s\s} \def\b{\!\!}
   \def\q{\quad\;}
   \hrule
   \halign{&\vrule#&\strut\;\hfil\s#\s\hfil\cr
   \h
   & $k$ && $c_7^{(k)}$ && $k$   && $c_7^{(k)}$  && $k$ && $c_7^{(k)}$ && $k$ && $c_7^{(k)}$ && $k$ && $c_7^{(k)}$ &\cr
   \h
   \noalign{\hrule}
   \h
   &  4 &&  2.13722  &&  7 &&  2.13348  && 10 &&  2.12209  && 13 && 2.13657 && 16 && 2.13628 &\cr\h
   &  5 &&  2.14004  &&  8 &&  2.15126  && 11 &&  2.13562  && 14 && 2.13599 && 17 && 2.13629 &\cr\h
   &  6 &&  2.14037  &&  9 &&  2.12815  && 12 &&  2.13597  && 15 && 2.13657 && 18 && 2.13646 &\cr\h
   \noalign{\hrule}
   }}
 }
}
\medskip
%
%

\begin{figure}
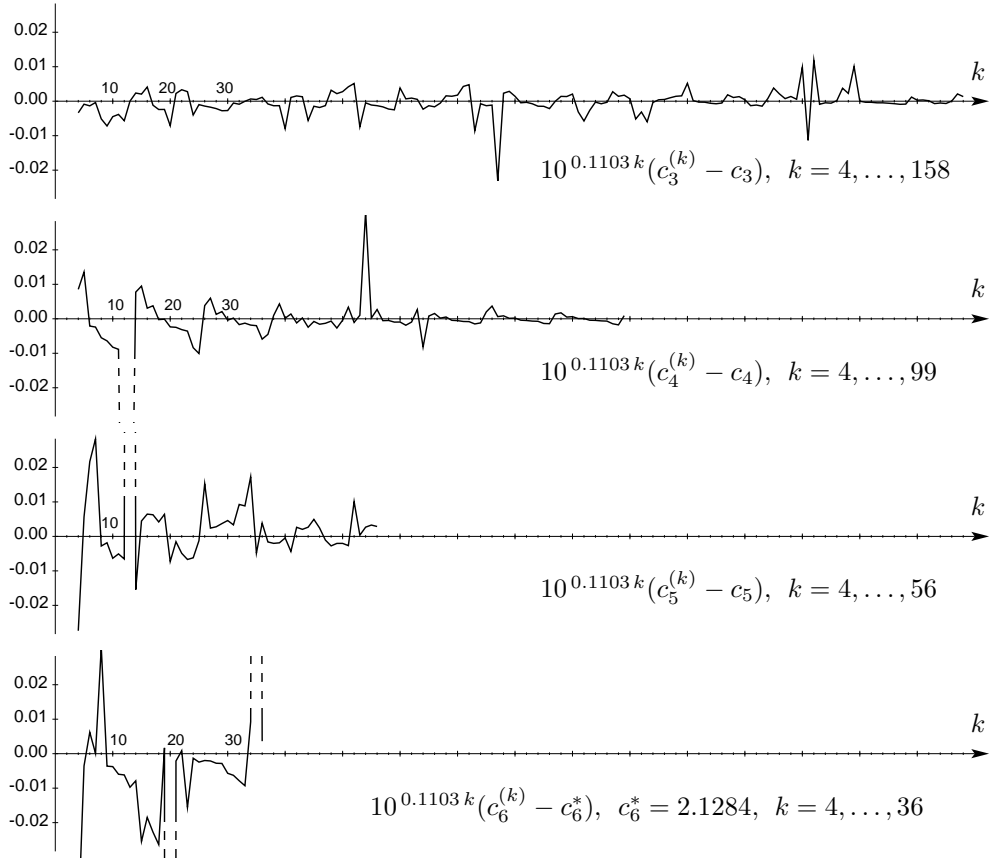

	\centering
	\includegraphics[width=130mm]{pade3.eps}
	\put(-8,47){$k$}
	\put(-170,10){$10^{\,0.1103\,k}(c_3^{(k)}-c_3)$, $\;k=4,\dots,158$}\vskip 3pt
	
	\centering\includegraphics[width=130mm]{pade4.eps}
	\put(-8,47){$k$}
	\put(-170,16){$10^{\,0.1103\,k}(c_4^{(k)}-c_4)$, $\;k=4,\dots,99$}\vskip 3pt
	
	\centering\includegraphics[width=130mm]{pade5.eps}
	\put(-8,47){$k$}
	\put(-170,16){$10^{\,0.1103\,k}(c_5^{(k)}-c_5)$, $\;k=4,\dots,56$}\vskip 3pt
	
	\centering\includegraphics[width=130mm]{pade6.eps}
	\put(-8,47){$k$}
	\put(-234,16){$10^{\,0.1103\,k}(c_6^{(k)}-c^*_6)$, $\;c^*_6=2.1284$, $\;k=4,\dots,36$}
\vskip-1mm
	\caption {Convergence of the empirical estimates to $c_3,c_4,c_5,c^*_6$.}
	\label{fig.emp}
\end{figure}

The second line in \eqref{eq.emp} leads to a conjecture that
$f(m,n)\sim \operatorname{const}\cdot\, 2^{c_m mn}/\sqrt n$ for any $m$.
A computation shows that the convergence of $\alpha_m^{(k)}$ to $-1/2$ looks as good as for the numbers $c_m^{(k)}$.
For $m=1$ this asymptotics follows from Stirling's formula:
$f(1,n)=\binom{2n}n\sim 4^n/\sqrt{\pi n}$.

Notice that usually such asymptotics mean that the first real singularity $x=x_0$ of
the generating function
is a brunching point and the leading term of the Laurent-Puiseux expansion at it
is $C(x-x_0)^{-1/2}$ (see, e.g., \cite[Figure~VII.24]{refFS}, \cite[p.~596]{refOtter}).

The factor $n^{-1/2}$ might also mean that there is a sort of Central Limit Theorem for the numbers
$j^*_{a,n-a}$ (defined \S\ref{sect.trapez4}) and their analogues for $m>4$.


\subsection{Convexity conjecture}

In \cite[\S2.3]{refOre2022} we formulated the following
{\rm convexity conjecture:} $f(m,n-1)f(m,n+1) \ge f(m,n)^2$
for all $m,n\ge 1$. It implies the bound
\begin{equation}\label{conv.conj}
     c_m \ge (n+1)c_{m,n+1} - nc_{m,n}
\end{equation}
for any $m,n$, in particular, $c\ge c_{115}\ge 5c_{115,5}-4c_{115,4}\approx 2.16848$.
The newly computed exact values of $f(m,n)$ still confirm this conjecture.

Passing to the limit in \eqref{conv.conj} we obtain $c\ge (n+1)c_{n+1}-nc_n$.
Using the computed values of $c_2,\dots,c_5$ the conjecture implies lower bounds
$$
   c\ge 3c_3-2c_2\approx 2.14641, \quad   
   c\ge 4c_4-3c_3\approx 2.16413, \quad   
   c\ge 5c_5-4c_4\approx 2.17436,         
$$
and 
$c\ge 6c_6^*-5c_5\approx 2.1803$ (see \S\ref{sect.empi}),    
which agrees with the expected bounds $2.2<c<2.3$ from \S\ref{sect.empi}
(recall that the best proven bounds are $2.118 < c < 2.777$).


\section{ Non-primitive lattice triangulations }\label{sect.np}

Denote the number of all (not necessarily primitive) lattice triangulations of the
$m \times n$ rectangle by $f^{\np}(m, n)$. 
In \cite[\S6]{refOre2022} we gave a rather coarse upper bound for these numbers.
In this section we prove the following asymptotic lower bound for them.
\begin{equation}\label{eq.np}
    \lim_{n\to\infty} f^{\np}(n,n)^{1/n^2} \ge 5.
\end{equation}
Indeed, given integers $n$ and $k\le n$, consider triangulations of the rectangle
$[0,n]^2$ such that each vertical segment $\{m\}\times[0,n]$, $m\in\Z$, is the
union of $k$ edges of the triangulation. There are $\binom{n}{k-1}$ ways to choose
vertices on each vertical line and, for fixed vertices, $\binom{2k}{k}$ triangulations of each vertical
strip (see \cite[Eq.~(2.1)]{refKZ}). Thus, the total number of such triangulations is
$\binom{n}{k-1}^{n+1}\binom{2k}{k}^n$. Let $f_x(n,n)$,  $x\in[0,1]$, be this number for $k=\lfloor xn\rfloor$.
By Stirling formula we obtain
$$
   \lim_{n\to\infty} \ln f_x(n,n)^{1/n^2} = x\ln 4 - x\ln x - (1-x)\ln(1-x).
$$
The first two derivatives of the right hand side are $\ln(4-4x)-\ln x$ and $1/(x-x^2)$, hence
the maximum is attained at $x=4/5$ and it is equal to $\ln 5$, whence the bound \eqref{eq.np}.

It is easy to check that the limit will be the same if we consider all lattice triangulations of $[0,n]^2$
such that each vertical strip of width 1 is a union of triangles.


\end{document}